\documentclass[10pt,letterpaper]{amsart}%
\usepackage{amssymb,amsmath,amsthm,amstext,amscd}%
\usepackage{graphicx,subfig,multirow}%
\usepackage[normalem]{ulem}%
\usepackage{url}%
\usepackage[usenames,dvipsnames,table]{xcolor}%
\usepackage{todonotes}%
\usepackage{siunitx, chemmacros}%
\usepackage{enumitem, paralist}%
 
\DeclareSIUnit\foot{ft}%
\usepackage{cleveref}%
\usepackage[left=1.1in, right=1.1in]{geometry}%
  
\title[Robust information divergences for random PDE]{Robust
  information divergences for model-form uncertainty arising from
  sparse data in random PDE}

\author{Eric Joseph Hall}%
\author{Markos A. Katsoulakis}%
\address{Department of Mathematics \& Statistics\\
  University of Massachusetts Amherst}%
\email{hall@math.umass.edu}%
\email{markos@math.umass.edu}%

\thanks{The work of all authors was supported by the Office of Advanced
  Scientific Computing Research, U.S.~Department of Energy, under
  Contract No.~DE-SC0010723.}%

\DeclareMathOperator{\E}{\mathbf{E}}%
\newcommand{\Pb}{P} \newcommand{\Qb}{Q}
\DeclareMathOperator{\var}{Var}%
 \DeclareMathOperator{\cov}{Cov}%
\newcommand{\dd}{\mathrm{d}}%
\DeclareMathOperator{\trace}{tr}%
\DeclareMathOperator{\RE}{\mathcal{R}}%
\DeclareMathOperator{\argmin}{arg\,min}
\newcommand{\indic}{\textbf{1}}%
\newcommand{\rset}{\textbf{R}}%

\newtheorem{theorem}{Theorem}%
\newtheorem{corollary}{Corollary}%
\newtheorem{lemma}{Lemma}%
\theoremstyle{plain}%
\newtheoremstyle{example}%
{}%
{}%
{}
{}
{\it}
{}
{\newline}
{\thmname{#1}\thmnumber{ #2}: \thmnote{#3}.}
\newtheoremstyle{remark}%
{}%
{}%
{}
{}
{\it}
{:}
{ }
{\thmname{#1}\thmnumber{ #2}}
\theoremstyle{remark} \newtheorem{remark}{Remark}%

\crefname{equation}{}{}%
\crefname{remark}{Remark}{Remarks}%
\crefname{theorem}{Theorem}{Theorems}%
\crefname{lemma}{Lemma}{Lemmas}%
\crefname{corollary}{Corollary}{Corollaries}%
\crefname{figure}{Figure}{Figures}%

\begin{document}%
\maketitle%

\begin{abstract}
  We develop a novel application of hybrid information divergences to
  analyze uncertainty in steady-state subsurface flow problems. These
  hybrid information divergences are non-intrusive, goal-oriented
  uncertainty quantification tools that enable robust, data-informed
  predictions in support of critical decision tasks such as regulatory
  assessment and risk management. We study the propagation of
  model-form or epistemic uncertainty with numerical experiments that
  demonstrate uncertainty quantification bounds for (i) parametric
  sensitivity analysis and (ii) model misspecification due to sparse
  data. Further, we make connections between the hybrid information
  divergences and certain concentration inequalities that can be
  leveraged for efficient computing and account for any available data
  through suitable statistical quantities.
\end{abstract}



\section{Introduction}%
\label{sec:introduction}

\subsection{Context and objectives}%
\label{sec:context-objectives}%

Stochastic modeling of complex systems involves a multilayered process
where various sources of uncertainty are accounted for at different
stages. In a goal-oriented framework, the ultimate focus of the model
is to estimate a quantity of interest (QoI) that depends on model
outputs in order to make predictions in support of risk management,
regulatory assessment, performance optimization, safety or reliability
engineering, and other critical decision tasks
(\cite{OberkampfEtAl:2004cp, DupuisChowdhary:2013ae}). Since model
outputs can be sensitive to distributional assumptions made throughout
this process, the analysis of uncertainty in a QoI is essential.
Additionally, applications in physically relevant settings often pose
challenges not present in the conceptual formulation that demand
incorporating data into the modeling process. The modeling of complex
systems in physically relevant settings, therefore, requires
goal-oriented uncertainty quantification (UQ) tools that allow one to
distinguish sources of uncertainty present in the system with a view
toward making robust, data-informed predictions.

\begin{figure}[tb]
  \centering
  \includegraphics[width=0.8\textwidth]{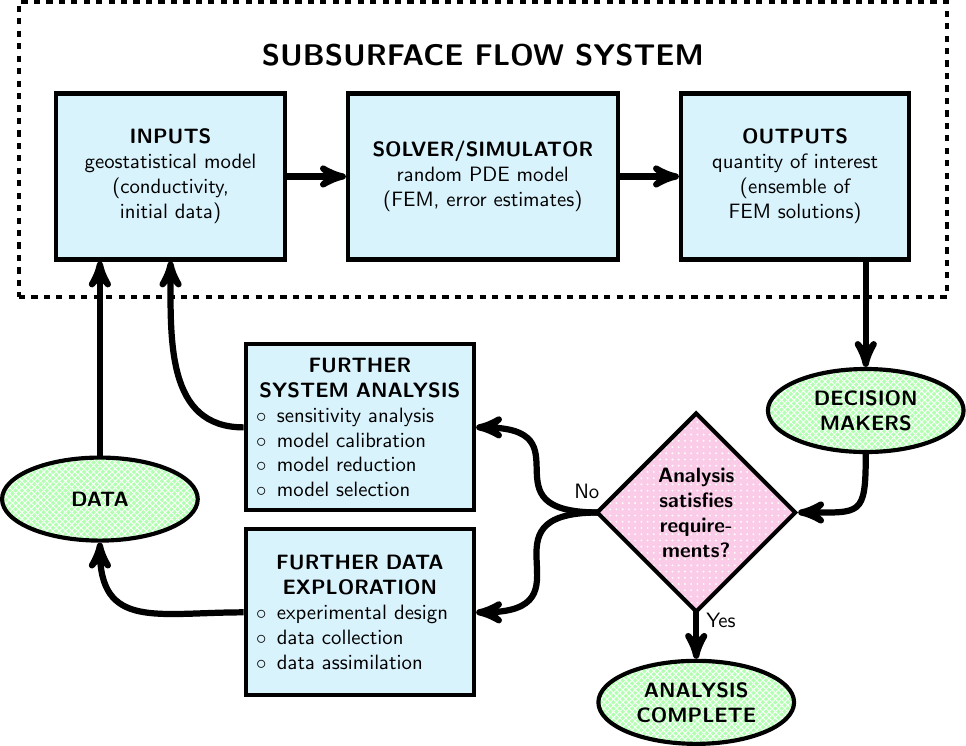}
  \caption{The modeling of steady-state subsurface flow in physically
    relevant settings demands goal-oriented UQ tools that allow one to
    distinguish different sources of uncertainty present in the system
    with a view toward making robust, data-informed predictions in
    support of decision tasks (adapted from
    \cite{OberkampfEtAl:2004cp}). A key challenge concerning the
    propagation of epistemic uncertainty from inputs to outputs is
    examined in detail in \cref{fig:propagation-uncertainty}.}
  \label{fig:layers-subsurface-flow-model}
\end{figure}

In the present work, we consider the modeling of steady-state
subsurface flow, a complex system depicted in
\cref{fig:layers-subsurface-flow-model} that has important
applications in hydrology, carbon sequestration, and petroleum
engineering (\cite{AndersonEtAl:2015gw, Ewing:1983rs,
  AarnesEtAl:2009ms, Dagan1989:ft}). The core mathematical problem
involves a random partial differential equation (PDE) of elliptic type
that describes the physics of steady-state flow. The stochastic
coefficient in the PDE represents a conductivity field given by a
geostatistical model with properties inferred from relevant data.
Robust predictions for a QoI, and in particular techniques for
quantifying the propagation uncertainty from the geostatistical model,
are critical to inform decision tasks. However, in physically relevant
settings the complexity of subsurface porosity and the sparsity of
available data pose a number of challenges.

We seek to address some of these challenges using techniques from
applied probability. Information divergences, the primary tool that we
will employ, have 
been successfully applied to problems in stochastic dynamics
(\cite{DupuisChowdhary:2013ae, DupuisEtAl:2015ps,
  HarmandarisEtAl:2016vi, GourgouliasEtAl:2016sp}). Based on the
Donsker--Varadhan variational principle (\cite{DupuisEllis:1997ld}),
information divergences provide goal-oriented UQ bounds that first
appeared in \cite{DupuisChowdhary:2013ae} and have since undergone
various extensions and further analysis (\cite{LiXiu:2012fp,
  AtarChowdharyDupuis:2015rd, KatsoulakisEtAl:2017sc}). However, the
application of information divergences to complex real-world systems,
for example, systems containing multiple scales, multiple levels of
fidelity, and multiple sources and types of uncertainty, have only
been shallowly explored. The multifaceted nature of the modeling task
for our problem of interest requires a formulation that distinguishes
multiples sources of uncertainty to which we attach varying levels of
confidence while also considering computational challenges to
practical implementation. Further, we wish to incorporate data into
this process in a systematic manner that complements existing
inference procedures and we provide new perspectives on the hybrid
information divergences that enable data-informed predictions.

For our system of interest, we develop a novel application of hybrid
information divergences to study the propagation of model-form
uncertainty related to the inputs of the goal-oriented framework in
\cref{fig:layers-subsurface-flow-model}. Model-form or epistemic
uncertainty (\cite{Helton:1994uq, HoffmanHammonds:1994pu, Rowe:1994uu,
  FersonGinzburg:1996uq, Hora:1996ae, Parry:1996uq,
  PateCornell:1996uq}) in this context expresses ignorance in the
nature of the geostatistical model due to the sparsity of available
data arising for example from lacking priors, incomplete information,
and missing science. This uncertainty represents a modeling error that
we evaluate as the weak error between a QoI obtained from a nominal
and an alternative geostatistical model for the inputs. These hybrid
information divergences build on traditional UQ by enabling bounds
across model predictions that would otherwise be computationally
expensive using standard techniques, e.g.\ Monte Carlo (MC) or
generalized polynomial chaos (gPC) expansions.

Importantly, the hybrid nature of the information divergences employed
here allows us to represent, aggregate, and distinguish various
sources of uncertainty by treating distinct sources under performance
measures that allow us to attach varying levels of confidence to
different parts of the model. These hybrid divergences have the form
of a relative entropy penalized by a risk-sensitive hybrid performance
measure, similar in flavor to the Gibbs variational formula in
statistical mechanics (\cite{Ellis:2006ed}), that strikes a balance
between data-dependent quantities (relative entropy) and 
physics-dependent quantities (risk-sensitive performance measures
i.e.\ cumulant generating functions that depend on the output of the
forward model). On the one hand, these bounds are robust in that for a
fixed nominal model they bound all alternative models within a given
information budget. On the other hand, these bounds are tight in that
there exists an alternative model within a given information budget
for which the bounds are attained as equality
(\cite{AtarChowdharyDupuis:2015rd,GourgouliasEtAl:2017aa}). We mention
that this information budget, when viewed as an allowable level of
model-form uncertainty, is philosophically similar to the privacy
level introduced in the context of R{\'e}nyi divergence in a model for
differential privacy (\cite{DuchiJordanWainwright:2013lp}). Numerical
experiments are included here to demonstrate the application of this
theory for (i) a straightforward UQ task in
\cref{sec:screening-and-sa} concerning parametric sensitivity analysis
and (ii) a more exploratory UQ task in \cref{sec:data-informed-bounds}
featuring bounds for model misspecification due to sparse data that
are not captured by small parametric perturbations.

Although a mathematically rigorous hybrid modeling framework was first
introduced in \cite{DupuisChowdhary:2013ae}, with variations in
\cite{LiQiXiu:2014uq}, here we explore the application of hybrid
goal-oriented divergences to stochastic models of complex systems
involving random PDE and also provide new insights on uncertainty
arising from data while providing a new perspective on the information
divergences as data-informed UQ bounds. In particular, we apply hybrid
information divergences to complex systems where the forward modeling
involves random PDEs where the random field is inferred from
real-world data and gPC cannot be used due to log-normality
(\cite{ErnstEtAl:2012pc}). We explore how the hybrid information
divergences quantify model-form uncertainty while also providing a
perspective linking model-form uncertainty to sparse data. This brings
new insights on the use of information divergence for predictive
modeling with data by studying how changes to data result in
alternative models that are not small perturbations of the nominal
model. Further, we explore connections with certain well-known
concentration inequalities (\cite{DemboZeitouni:2010ld}) that can be
leveraged for efficient computing; this latter approach was recently
introduced and applied to model problems in
\cite{GourgouliasEtAl:2017aa} and we extend these ideas here to models
for complex systems involving random PDE where the stochastic fields
are infinite dimensional and inferred from real-world data. The
concentration inequalities yield non-intrusive upper and lower
uncertainty quantification bounds that involve statistical quantities
more easily computed than the cumulant generating functions in earlier
works on information divergences.

The hybrid divergences presented here provide a UQ framework that is
well suited to the particular application of interest for a number of
reasons. Firstly, hybrid information divergences provide a natural way
of representing and distinguishing various sources of uncertainty
arising in the distinct layers of the modeling process for complex
systems that complement existing inference procedures while allowing
us to attach varying levels of confidence in different components of
our model. Secondly, the structure of the hybrid divergences allows UQ
computations to be carried out non-intrusively by interfacing with
existing simulation and sampling methods for the random PDE such as
the popular MC finite element method (FEM). Thirdly, the hybrid
information divergences encapsulate ``worst case'' scenarios within a
rigorous formulation similar to a robust optimization perspective
(\cite{Ben-TalGhaouiNemirovski:2009ro}) making them appropriate
deliverables in the context of decision support. In the remainder of
this section, we formulate the model problem and some of the main UQ
challenges that motivate our approach.

\subsection{Formulation of the model problem}%
\label{sec:model-problem}%

Presently we detail the main layers comprising the subsurface flow
system in \cref{fig:layers-subsurface-flow-model} each in turn.

\subsubsection{Random PDE model}
\label{sec:random-pde-model}
For a given probability space $(\Omega, \mathcal{F}, \Pb)$, we
consider the random PDE in the unknown $u$,
\begin{equation}
  \label{eq:model-rpde}
  - \nabla \cdot ( a(\omega, x) \nabla u(\omega, x)) = f(x), \quad
  \text{for } x \in \Gamma \subset \rset^d,
\end{equation}
subject to $u = u_0$ on $\partial \Gamma$, with given data $u_0$,
source term $f: \Gamma \to \rset$, and stochastic conductivity
$a:\Omega \times \Gamma \to \rset$. Arising from Darcy's law with
continuity, \cref{eq:model-rpde} is a model for steady-state flow or
diffusion. In subsurface hydrology, problem~\cref{eq:model-rpde}
models time-independent groundwater flow where $u$ might represent a
water head (pressure) or the concentration of a containment
(\cite{AndersonEtAl:2015gw, Dagan1989:ft}). The existence of a unique
pathwise variational solution $u$ to \cref{eq:model-rpde} follows by
assuming sufficiently regular data and boundedness of the conductivity
(\cite{Charrier:2012}). For simplicity of presentation we consider the
steady-state system \cref{eq:model-rpde}, however, the theory that
follows can be applied to other forward models for subsurface flow
and/or transport such as (reaction-)advection-dispersion equations
\cite{Tartakovsky2013:rv,AndersonEtAl:2015gw}.

In the sequel we consider the standard, continuous piecewise-linear
FEM approximation $\bar{u}(\omega) \approx u(\omega)$ of the pathwise
variational solution of \cref{eq:model-rpde} and then sample QoIs via
a MC method. This approach is favored for the application of interest
since the finite element solution $z = (\bar{u}_n)$ is a possibly high
dimensional random vector due to the low-regularity of the
conductivity field. A discrete projection of the conductivity field
$y = (\bar{a}_n) \approx a(\omega)$ is then used when forming the
stiffness matrix in the FEM solver. Other approaches, such as the
stochastic Galerkin (\cite{MatthiesKeese:2005}) or stochastic
collocation (\cite{BabuskaNobileTempone:2007}) methods are typically
advantageous when the conductivity possesses more regularity; being
non-intrusive, our approach is also applicable using these methods.
Regarding the analysis of FEM errors, \emph{a priori} estimates are
available in \cite{CharrierScheichlTeckentrup:2013} and computable
goal-oriented estimates for problems with rough stochastic
conductivities are in \cite{Halletal:2016sc}.

\subsubsection{QoIs}
\label{sec:qois}
The goal of problem \cref{eq:model-rpde} is to estimate a QoI,
\begin{equation*}
  \E_\Pb [g(u)] = \int_\Omega g(u(\omega,x)) \dd\Pb,
\end{equation*}
for a given goal functional $g$ in support of a decision task. The
numerical experiments in this work focus on goal functionals that
yield statistics of point estimates, $g(u) = u(x)$, and indicator goal
functionals, $g(u)= \indic_A$ that correspond to failure
probabilities,
\begin{equation}
  \label{eq:failure-probability}
  \E_\Pb[\indic_{A}] = \Pb(A)\,,
\end{equation} 
for events $A \subset \Omega$ such as
$A :=\{\omega : u(\omega, x_0) > k \}$, the event that the solution at
$x_0 \in \Gamma$ exceeds a threshold $k$. In this instance, QoIs might
involve solving the steady-state system \cref{eq:model-rpde} for
contaminant concentrations and hydraulic heads or fluxes (e.g. heads
$u(x)$, fluxes $q(x) := - a(x) \nabla u(x)$) to predict QoIs or
forecast future conditions (\cite{AndersonEtAl:2015gw}). QoIs such as
failure probabilities, for example, the probability of a contaminant
exceeding a threshold or a contaminant plume reaching a specified
zone, would then factor into larger probabilistic risk assessment,
i.e.\ concerning contamination of groundwater, that would represent
perhaps one system fault or route to failure
(\cite{Tartakovsky2013:rv}). A QoI that is meant to provide a
quantitative description of the system can be highly sensitive to
distributional assumptions on the underlying geostatistical model.
This poses a key challenge for the support of decision tasks and
highlights the importance of understanding the propagation of modeling
error due to the geostatistical model in order to quantify uncertainty
in a QoI.

\subsubsection{Geostatistical model}%
\label{sec:geostatistical-model}%
Subsurface materials are observed to be heterogeneous over each of the
problem scales related to experimental measurements
(\cite{Dagan:1986}). In applications to physically relevant settings,
fully resolving a model for $a$ requires more data than is possible to
acquire. Uncertainty in the problem data is subsequently captured
through a geostatistical model for $a$ that typically takes the form
of a log-normal random field that possesses low regularity. Such a
field is characterized by its mean, $\mu(x) = \E[\log a(x)]$, and
covariance,
\begin{equation*}
  \label{eq:spatial-correlations}
  C(x,\tilde{x}) = \cov[\log a(x), \log a(\tilde{x})], 
  \qquad \text{for } x, \tilde{x} \in \Gamma.
\end{equation*}
These quantities describe the spatial structure in terms of statistics
between different locations and are nontrivial to model for
heterogeneous media. In applications, these quantities are typically
further assumed to have a parametric form where the hyperparameters
that describe $\mu$ and $C$ are representative of physical properties
that can theoretically be measured using relevant data. Inference
procedures for the hyperparameters range from classical geostatistical
methods, that rely on fitting variograms using likelihoods or moments,
to more advanced Bayesian methodologies (\cite{GelfandEtAl:2010hb}).
For applications in physically relevant settings, the complex physics
of subsurface porosity results in sparse data that diminishes our
confidence in the form of the geostatistical model. In the next
section, we further motivate and formulate these challenges.

\subsection{UQ challenges}%
\label{sec:uq-challenges}%

The rough log-normal random fields that are typically used to capture
the heterogeneity of subsurface materials produce computational
challenges at the level of the solver and when analyzing model outputs
(\cref{fig:layers-subsurface-flow-model}). For example, gPC methods
exhibit slow convergence due to the log-normality
(\cite{ErnstEtAl:2012pc}). As noted in
\cite{CliffeGilesScheichlTeckentrup:2011}, the correlation lengths
involved in the geostatistical model are typically short with respect
to the problem domain, hence stochastic Galerkin methods yield high
dimensional QoIs, but still too large to guarantee a separation of
scales, an obstacle to stochastic homogenization techniques. In
addition, the complex physics also affects the acquisition and
reliability of data.

The data used to identify the salient features of the geostatistical
model are typically sparse due to costs arising from a number of
confounding factors. In this context, data may come from either
empirical or experimental measurements collected by domain scientists
at different scales of the problem. Then, relevant data, for instance,
the permeability field, needs to be assimilated and inferred from
these measurements. For example, a geostatistical model for the
conductivity might be based on empirical porosity and water retention
measurements from a controlled laboratory experiment in combination
with variogram fitting methods (\cite{Durner:1994hc}) or on hydraulic
head measurements from \emph{in situ} field tests in combination with
a Bayesian inverse problem (\cite{McLaughlinTownley:1996ip}). In both
cases, the geostatistical model is determined from data whose
availability is limited and whose reliability should be questioned.
The combination of these factors influences our confidence in the form
of the geostatistical model.

Therefore, we view the geostatistical model as a source of model-form
or epistemic uncertainty. Although theoretically reducible,
eliminating epistemic uncertainty entirely for subsurface flow is not
feasible due to the prohibitive cost of collecting sufficient data.
Additionally, the physical system has other sources of randomness such
as aleatoric uncertainty, or variability, in the model inputs, solver,
and outputs and may have independent sources of epistemic uncertainty
that arise in each of these layers. All of this randomness propagates
to the QoI and influences the decision task. A less refined approach,
in contrast to the one considered here, is to assume that epistemic
uncertainty can be modeled by aleatoric uncertainty which is typically
the case in a standard MC approximation when one assumes that a
distribution for each uncertain aspect of the system exists
(\cite{OberkampfEtAl:2004cp, DupuisChowdhary:2013ae}).

The main UQ challenges postulated above are summarized as follows:
\begin{compactitem}[\qquad $\circ$]
\item represent and distinguish various sources of uncertainty in the
  system;
\item propagate model-form uncertainty;
\item inform decision tasks through robust data-informed deliverables;
\item quantify the impact of sparse data on predictions in a
  goal-oriented framework; and
\item analyze heterogeneity of subsurface physics.
\end{compactitem}
The final challenge point, related to the solver layer of
\cref{fig:layers-subsurface-flow-model}, leads to considerations that
impact the random PDE model, such as UQ for multi-phase flows, and are
beyond the scope of the present work. The focus of the present work is
instead toward addressing the first four challenge points, that have
important implications for the inputs, outputs, further analysis, and
data layers of the model in \cref{fig:layers-subsurface-flow-model}.
In particular, we address the propagation of model-form uncertainty by
employing hybrid representations of information divergences,
introduced in the next section, that allow us to represent and
distinguish various sources of randomness.

\section{Hybrid information divergences and model-form uncertainty}%
\label{sec:hybr-inform-diverg}%

\subsection{Propagation of model-form uncertainty}
\label{sec:prop-model-form}

A key challenge addressed in this work concerns the propagation of
model-form uncertainty within the goal-oriented framework in
\cref{fig:layers-subsurface-flow-model}. We view the propagation of
model-form uncertainty as a modeling error or model bias,
\begin{equation}
  \label{eq:weak-error-observables}
  \mathcal{E}(\Qb, \Pb; g(\bar{u})) 
  := \E_\Qb[g(\bar{u})] - \E_\Pb[g(\bar{u})],
\end{equation}
the weak error between a QoI evaluated under a nominal measure $\Pb$
and an alternative measure $\Qb$. We do not give a method for choosing
the nominal model $\Pb$ or the alternative model $\Qb$, but only a
method for reasoning quantitatively about the effect that different
modeling choices will have on decision support. In this context we
view the nominal model $\Pb$ as a computationally tractable model or
surrogate presently being used to inform decision tasks; in subsurface
flow, the nominal model would be derived from the conceptual model
chosen by a geoscientist that summarizes what is known about the
hydrogeological system
(\cite{AndersonEtAl:2015gw,Tartakovsky2013:rv}). Respectively, the
alternative model $\Qb$ can either be viewed as a fine scale model
that is not computationally tractable (for instance $\Pb$ as a
coarse-grained model for $\Qb$) or, in view of sparse data, as another
surrogate that is also plausible ($\Qb$ equally as plausible as the
nominal model $P$). In either case, the weak error
\cref{eq:weak-error-observables} suggests how different modeling
choices impact model predictions of the QoI. We seek representations
for $\Pb$ and $\Qb$ that will allow us to track the propagation of
uncertainty from model inputs to outputs as illustrated in
\cref{fig:propagation-uncertainty}. We observe that although the
propagation of model-form uncertainty is controlled by
\cref{eq:weak-error-observables}, quantifying the propagation directly
using \cref{eq:weak-error-observables} will be computationally
infeasible (cf.\ \cref{cor:info-budget} below).

\begin{figure}
  \centering
  \includegraphics[width=0.8\textwidth]{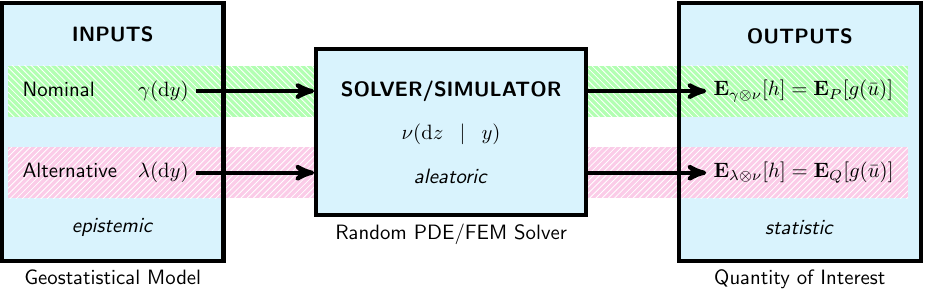}
  \caption{Detail of the dashed box in
    \cref{fig:layers-subsurface-flow-model} that demonstrates
    model-form uncertainty in the geostatistical model propagating
    through the random PDE solver, with distribution
    \cref{eq:solution-distribution}, to a QoI. The propagation is
    controlled by the weak error \cref{eq:weak-error-observables}
    between a QoI evaluated under a nominal and alternative model with
    distributions \cref{eq:nominal-alternative-distributions}.}
  \label{fig:propagation-uncertainty}
\end{figure}

As the conductivity field $a$ is a source of epistemic uncertainty,
the distribution of $y = (\bar{a}_n)$ is not known and it will be of
natural interest to study the propagation of this source of
uncertainty in particular and to compare the effects of different
solver inputs as suggested in \cref{fig:propagation-uncertainty}. To
this end, we consider
\begin{equation}
  \label{eq:nominal-alternative-distributions}
  \text{nominal } y \sim \gamma \quad\text{or}\quad 
  \text{alternative } y \sim \lambda,
\end{equation}
for measures $\gamma$ and $\lambda$ on a Polish space $\mathcal{Y}$,
arising from the geostatistical models proposed for $a$. For example,
the nominal and alternative distribution of $y$ could be associated
with small parametric perturbations or to completely different models
within a parametric family. The discrete solution $z = (\bar{u}_n)$,
resulting from a given algorithm for solving the random PDE,
represents a source of variability as it is modeled by a stochastic
process. Theoretically, $z$ has a known distribution given by $\nu$,
on a Polish space $\mathcal{Z}$, that depends on the solver and
$\bar{a}$. In general, even if the distribution of $a$ is prescribed
it may not be tractable to sample from $\nu$ directly. Instead, the
conditional distribution of $\bar{u}$ given $y$, that is,
\begin{equation}
  \label{eq:solution-distribution}
  z = (\bar{u}_n) \sim \nu(\dd z \mid y),
\end{equation}
is a computable quantity that we will evaluate using a FEM solver as
indicated in \cref{fig:propagation-uncertainty}. Theoretically,
qualitatively different distributions for $y$ (i.e.\ geostatistical
models with more or less regularity) could necessitate different
solvers; in this case, one could also include $\nu'$ different from
$\nu$. For simplicity here we assume that the same solver with
distribution $\nu(\dd z \mid y)$ is used with respect to the nominal
and alternative model. Stating our objective more formally, we
consider \cref{eq:weak-error-observables} between a QoI sampled with
respect to the full Bayesian network for the nominal measure
$\Pb = \gamma \otimes \nu$ and alternative measure
$\Qb = \lambda \otimes \nu$. The preceding analysis suggests that
bounds for model outputs that distinguish these different sources of
uncertainty can be obtained by concentrating on a special case of the
hybrid bounds proposed in \cite{DupuisChowdhary:2013ae} that rely on
the conditional distribution \cref{eq:solution-distribution}.

\subsection{Representing and distinguishing sources of uncertainty}%
\label{sec:hybr-perf-meas}%

To achieve the desired bounds, we represent $g(\bar{u})$ in terms of a
variational form
\begin{equation*}
  g(\bar{u}) = h(y,z)
\end{equation*}
that distinguishes between the epistemic variable $y$, related to
$\bar{a}$, and aleatoric variable $z$, related to $\bar{u}$. That is,
for $\Qb = \lambda \otimes \nu$ we define a hybrid performance measure
$h$ by
\begin{equation}
  \label{eq:performance-meas}
  \E_{\lambda\otimes\nu} [h] 
  = \int_{\mathcal{Y}}\int_{\mathcal{Z}} h(y,z) 
  \nu(\dd z \mid y) \lambda(\dd y) 
  = \int_\Omega g(\bar{u}) \dd \Qb 
  = \E_\Qb [g(\bar{u})].
\end{equation}
For simplicity in the presentation we assume that $z = (\bar{u}_n)$ is
only a source of variability and that the $u_0$ and $f$ in
\cref{eq:model-rpde} are deterministic; however, observe that the
representation \cref{eq:performance-meas} can easily be extended to
characterize various uncertain aspects of the system by introducing
multiple integrals that aggregate and distinguish each independent
source of randomness.

A key observation is that the hybrid performance measure
\cref{eq:performance-meas} can then be expressed as
\begin{equation}
  \label{eq:equality-h-Hg}
  \E_{\lambda\otimes\nu} [h] = \E_\lambda [H^g]\,,
\end{equation}
where the random variable $H^g(y)$ is the marginal performance measure
given by,
\begin{equation}
  \label{eq:hybrid-pm}
  H^g(y) = \int_\mathcal{Z} h(y,z) \nu(\dd z \mid y) \,,
\end{equation}
for $y \in \mathcal{Y}$. $H^g$ encodes the propagation of the
model-form uncertainty to the ensemble solution of the random PDE (the
outputs in \cref{fig:propagation-uncertainty}). Approximation of a QoI
depending on the goal functional $g$ amounts to a standard MC estimate
for the sample mean of $H^g$ where $H^g$ is computed in a
non-intrusive manner using any available algorithm for the solver.
Next, we introduce an important information theoretic tool that will
allow us to measure differences between models suggested for epistemic
variables.

\subsection{Relative entropy in UQ}
\label{sec:relative-entropy-uq}

The relative entropy, or Kullback--Leibler divergence, quantifies the
discrepancy between two distributions (see for example \cite[\S
A.5]{RasmussenWilliams:2006gp}). Given probability measures $\gamma$
and $\lambda$, such that $\lambda \ll \gamma$, the relative entropy of
$\lambda$ with respect to $\gamma$ is
\begin{equation*}
  \label{eq:re-def}
  \RE(\lambda \mid \gamma) = 
  \int \log\frac{\dd \lambda (y)}{\dd \gamma (y)} \lambda(\dd y), 
\end{equation*}
and we note that $\RE(\lambda \mid \gamma) \geq 0$ with equality if
and only if $\lambda = \gamma$ almost everywhere (the Gibbs
inequality). For distributions that belong to a general exponential
family, closed formulas exist for the relative entropy
(\cite{GilEtAl:2013rd,LieseVajda:2007sd}). For example, the relative
entropy for $d$-dimensional multivariate Gaussian distributions
$\lambda = \mathcal{N}(\mu_i, \Sigma_i)$ and
$\gamma = \mathcal{N}(\mu_j, \Sigma_j)$ is given by
\begin{equation}
  \label{eq:re-multivar-norm}
  \RE(\lambda \mid \gamma) 
  = \frac{1}{2}\left(\log |\Sigma_j| - \log |\Sigma_i|
    + \trace\left(\Sigma_j^{-1}\Sigma_i\right) 
    +  (\mu_i - \mu_j)^\top \Sigma_j^{-1} (\mu_i - \mu_j) - d\right),
\end{equation}
where $|\cdot|$ denotes the
determinant. 

The relative entropy is related to the observable $H^g$ via the
Legendre transform of the cumulant generating functional
$\log \E_\gamma [e^{H^g}]$. We will use this well known fact from
Large Deviations Theory to study the propagation of model-form
uncertainty in \cref{fig:propagation-uncertainty}. This representation
balances data- and physics-based components and in particular we use
this new perspective in \cref{sec:data-informed-bounds} to provide
data-informed UQ bounds. The proof of the following lemma is well
known and available in \cite{DupuisEllis:1997ld}, nevertheless we
provide it for the convenience of the reader.

\begin{lemma}
  \label{lem:bound-variational-form}
  Let the marginal performance measure $H^g$ \cref{eq:hybrid-pm} be
  measurable and bounded and let $\gamma$ be a probability measure on
  $(\Omega, \mathcal{F})$. Then
  \begin{equation*}
    \log \E_{\gamma} [ e^{H^g}] 
    = \sup_{\lambda \ll \gamma} \left\{ 
      \E_\lambda [H^g] - \RE(\lambda \mid \gamma) \right\}.
  \end{equation*}
\end{lemma}

\begin{proof}
  Define $\nu$ as
  $\dd \nu / \dd \gamma = e^{H^g} / \E_\gamma [e^{H^g}]$. Then
  \begin{align*}
    - \RE(\lambda \mid \gamma) + \E_\lambda [H^g] 
    &= 
      - \E_\lambda \!\!\left[\log \frac{\dd\lambda}{\dd\gamma}
      \right] 
      + \E_\lambda [H^g]\\
    &= -\E_\lambda \!\!\left[\log \frac{\dd\lambda}{\dd\nu}
      \right] 
      - \E_\lambda \!\!\left[\log\frac{\dd\nu}{\dd\gamma}
      \right] 
      + \E_\lambda [H^g]\\
    &= -\RE(\lambda \mid \nu) + \log \E_\gamma[e^{H^g}].
  \end{align*}
  This proves the lemma and $\nu$ is necessarily the suprimizing
  measure.
\end{proof}

\Cref{lem:bound-variational-form} is the key ingredient in the proof
of \cref{thm:information-divergence}, concerning hybrid information
divergences, that follows in the next section.

\subsection{Goal-oriented hybrid divergences for modeling error}
\label{sec:goal-oriented-hybrid-div}

We begin by defining a hybrid risk-sensitive performance measure,
\begin{equation}
  \label{eq:risk-sensitive-pm}
  \Lambda_\gamma (c, H^g) := \frac{1}{c} \log \int_{\mathcal{Y}}
  e^{c\left(\int_{\mathcal{Z}} h(y,z) \nu(\dd z \mid y) - \E_\gamma
      [H^g] \right)} \gamma(\dd y) 
  = \frac{1}{c} \log \E_\gamma [e^{c(H^g - \E_\gamma[H^g])}],
\end{equation}
for a parameter $c >0$ that depends on the goal function $g$ through
the marginal performance measure $H^g$. As suggested by the last
equality in \cref{eq:risk-sensitive-pm}, $\Lambda_\gamma$ is a
weighted cumulant generating functional (log moment generating
functional) for the centered random variable
$c(H^g - \E_\gamma[H^g])$. The idea behind
\cref{thm:information-divergence} below is to apply
\cref{lem:bound-variational-form} to \cref{eq:risk-sensitive-pm} to
obtain a family of bounds indexed by $c$. The best bound is then found
by minimizing over $c$,
\begin{equation}
  \label{eq:xi}
  \Xi (\lambda \mid \gamma; H^g) :=
  \inf_{c>0} \left\{ \Lambda_\gamma(c,H^g) 
    + \frac{1}{c} \RE(\lambda \mid \gamma)\right\}\,,
\end{equation}
which defines a goal-oriented hybrid information divergence. This
insight was first made in \cite{DupuisEtAl:2015ps} where information
divergences were applied to stochastic dynamical systems with single
sources of uncertainty; we contrast the naive implementation of the
information divergences from \cite{DupuisEtAl:2015ps} with the hybrid
information divergences developed here for our problem of interest
after the proof of \cref{thm:information-divergence}. That
\cref{eq:xi} is a divergence, i.e.\
$\Xi(\lambda \mid \gamma ; H^g) \geq 0$ and
$\Xi(\lambda \mid \gamma ; H^g) = 0$ if and only if $\lambda = \gamma$
almost everywhere or $H^g$ is constant $\gamma$-a.s., follows
analogously to the proof of Theorem~2.7 in \cite{DupuisEtAl:2015ps}. A
bound for the weak error \cref{eq:weak-error-observables} specialized
to \cref{eq:xi} is formulated as follows.

\begin{theorem}[Hybrid Information Divergence]
  \label{thm:information-divergence}
  For a probability measure $\gamma$, measurable marginal performance
  measure $H^g$, and finite $\Lambda_\gamma(c, H^g)$ in a neighborhood
  about $c=0$, the bound
  \begin{equation}
    \label{eq:xi-bound}
    -\Xi (\lambda \mid \gamma; -H^g) \leq 
    \E_{\Qb}[g(\bar{u})] 
    - \E_{\Pb}[g(\bar{u})] 
    \leq \Xi (\lambda \mid \gamma; H^g)
  \end{equation}
  holds for any probability measure $\lambda$ such that
  $\RE(\lambda \mid \gamma) < \infty$ where
  $\Qb = \lambda \otimes \nu$ and $\Pb = \gamma \otimes \nu$. Further,
  the goal-oriented divergence can be linearized with respect to the
  relative entropy,
  \begin{equation}
    \label{eq:xi-linearization}
    \Xi (\lambda\mid \gamma; \pm H^g) = 
    \sqrt{\var_\gamma\left[H^g\right]}\sqrt{2\RE(\lambda\mid\gamma)} 
    + O(\RE(\lambda\mid\gamma))\,.
  \end{equation}
\end{theorem}

\begin{proof}
  For any bounded and measurable observable $H^g$, replacing $H^g$ in
  \cref{lem:bound-variational-form} with $c(H^g - \E_\gamma[H^g])$,
  for $c > 0$, yields
  \begin{equation*}
    \log \E_\gamma [e^{c(H^g-\E_\gamma[H^g])}] 
    = \sup_{\lambda \ll \gamma} \left\{
      c (\E_\lambda [H^g] - \E_\gamma [H^g])
      - \RE (\lambda \mid \gamma) \right\}.
  \end{equation*}
  This variational characterization yields the bounds
  \begin{subequations}
    \begin{align}
      & \E_\lambda[H^g] - \E_\gamma[H^g] \leq 
        \frac{1}{c} \log \E_\gamma [e^{c(H^g-\E_\gamma[H^g])}] 
        + \frac{1}{c} \RE(\lambda \mid \gamma) 
        \quad \text{and} \label{eq:direct-upper-bound}\\
      & \E_\lambda[H^g] - \E_\gamma[H^g] \geq 
        -\frac{1}{c} \log \E_\gamma [e^{c(H^g-\E_\gamma[H^g])}] 
        - \frac{1}{c} \RE(\lambda \mid \gamma)
        \label{eq:direct-lower-bound}
    \end{align}
  \end{subequations}
  where the optimal bound can be found by minimizing
  \cref{eq:direct-upper-bound} (respectively maximizing
  \cref{eq:direct-lower-bound}) over $c > 0$. This implies
  \cref{eq:xi-bound} for any measurable and bounded $H^g$. This bound
  can be extended to any measurable $H^g$, by considering
  \cref{eq:direct-upper-bound,eq:direct-lower-bound} with
  $H^g_{a,b} - \E_\gamma[H^g]$ in place of $H^g - \E_\gamma[H^g]$ in
  the cumulant generating functional, where
  $H^g_{a,b} = [H^g \vee (-a)] \wedge b$ for $a,b \in \rset$, and then
  letting $a \to \infty$ under the monotone convergence theorem and
  $b\to \infty$ using the dominating function
  $e^{c(H^g - \E_\gamma[H^g])}$, following the argument given in
  \cite[p.~86]{DupuisEtAl:2015ps}. The linearization
  \cref{eq:xi-linearization} arises from an asymptotic expansion at
  $\RE(\lambda \mid \gamma) = 0$, i.e.\ when $\lambda$ is a
  perturbation of $\gamma$, which relies on first proving that the
  optimization problems in \cref{eq:xi-bound} admit a unique solution
  $c^*(\rho)$. The details follow analogously from the proofs of Lemma
  2.11 and Theorem 2.12 in \cite{DupuisEtAl:2015ps}.
\end{proof}

\Cref{thm:information-divergence}, suggests an upper and lower bound
for the weak error \cref{eq:weak-error-observables} between model
predictions with respect to the full Bayesian network formulation for
the nominal model $P$ and alternative model $Q$. \Cref{eq:xi-bound}
trades the problem of sampling the weak error
\cref{eq:weak-error-observables}, that involves sampling the QoI with
respect to both $\Qb$ and $\Pb$, for an optimization problem
\cref{eq:xi} that only requires sampling a cumulant generating
functional with respect to $\Pb$. This enables comparisons among
models in contrast to the optimal uncertainty bounds in
\cite{OwhadiEtAl:2013uq}. In the present work, we do not compare model
outputs to experimentally observed flow variables nor do we assimilate
data from experimental observations to re-calibrate the probabilistic
models in the experiments. Due to the sparsity of available data,
model verification is not typically employed in geostatistical
modeling as all available data is used to formulate and calibrate the
model (\cite{DohertyHunt:2010ap}).

We reiterate that bounds \cref{eq:xi-bound} are non-intrusively
computable and tight. The structure of the marginal performance
measure $H^g$ allows the cumulant generating functional appearing in
\cref{eq:xi-bound} to be sampled non-intrusively (i.e.\ without
changing the code for simulating the forward model) with respect to
the nominal model $\gamma$ (the proposed computationally tractable
model). Recall that $H^g$ \cref{eq:hybrid-pm} is intimately related to
the propagation of epistemic uncertainty in
\cref{fig:propagation-uncertainty} from the inputs to the
solver/simulator. In contrast to classical bounds derived from the
Pinsker or Chapman--Robbins inequalities (\cite{CoverThomas:2006in,
  Tsybakov:2009np}), \cref{eq:xi-bound} is tight in that equality is
attainable for a given QoI by a suitable $\lambda$ within a given
relative entropy distance of the nominal model, as suggested in
\cref{fig:bounds-across-family} (for a full discussion on tightness
see
\cite{GourgouliasEtAl:2017aa}).

In general, the hybrid information divergence \cref{eq:xi} allows one
to express different levels of confidence in various components in a
complex system and we refer to these bounds as hybrid divergences
following the terminology in \cite{DupuisChowdhary:2013ae}. For
complex subsurface flow systems
(\cref{fig:layers-subsurface-flow-model}), \cref{eq:xi} possess a form
aligned with our goal of evaluating model-form uncertainty as a
modeling error. In contrast, the goal-oriented information divergences
from \cite{DupuisEtAl:2015ps} defined on product measures has the form
\begin{equation}
  \label{eq:xi-orig}
  \Xi(\lambda\otimes\nu \mid \gamma\otimes\nu; h) := \inf_{c>0}
  \left\{ \Lambda_{\gamma\otimes\nu}(c,h) + \frac{1}{c}
    \RE(\lambda \otimes\nu \mid \gamma \otimes\nu)\right\}\,,
\end{equation}
where $\Lambda_{\gamma\otimes\nu}$ is the standard risk-sensitive
performance measure given by
\begin{equation}
  \label{eq:standard-pm}
  \Lambda_{\gamma\otimes\nu}(c,h) = \frac{1}{c} \log
  \E_{\gamma\otimes\nu}[e^{c(h-\E_{\gamma\otimes\nu}[h])}]\,.
\end{equation}
The divergence \cref{eq:xi-orig} incorporates the epistemic and
aleatoric variables in a balanced manner that does not conform with
the asymmetrical way that we view these different sources of
uncertainty. Moreover, in the present setting it may not be possible
to sample $h$ directly whereas the marginal performance measure $H^g$
is the natural quantity to sample in the context of the MC FEM
approach. To make the connection between the original divergence and
hybrid divergence concrete, we observe that \cref{eq:xi-orig} depends
on the representation for $h$ while \cref{eq:xi} depends on $H^g$
(cf.\ \cref{eq:equality-h-Hg}). While both \cref{eq:xi-orig} and
\cref{eq:xi} yield valid bounds for the weak error
\cref{eq:weak-error-observables}, we have that
\begin{equation*}
  \Xi (\lambda \mid \gamma; H^g) \le \Xi(\lambda\otimes\nu \mid
  \gamma\otimes\nu; h) = \Xi (\Qb \mid \Pb; g)
\end{equation*}
due to Jensen's inequality applied to the exponential of
\cref{eq:hybrid-pm}. That is, the naive implementation of
\cref{eq:xi-orig} contains (not surprisingly) more uncertainty than
the hybrid divergence \cref{eq:xi}.

In the sequel, we write
\begin{equation*}
  \Xi_{+} := \Xi\, , 
  \quad \text{and} \quad
  \Xi_{-}(\cdot \mid \cdot; H^g) := -\Xi(\cdot \mid \cdot; -H^g)
\end{equation*}
to have a short notation for distinguishing the upper bound from the
lower bound. Next, we emphasizes that \cref{eq:xi-bound} applies to
all alternative models within a given information budget.

\begin{corollary}
  \label{cor:info-budget}
  Let the assumptions of \cref{thm:information-divergence} hold and
  let $\rho := \RE(\lambda \mid \gamma)$. Then
  \begin{equation}
    \label{eq:uq-bounds-rho}
    - \inf_{c>0} \left\{ \Lambda_\gamma(c,-H^g) + \frac{\rho}{c} \right\} 
    \leq \E_\Qb[g(\bar{u})] - \E_\Pb[g(\bar{u})] \leq 
    \inf_{c>0} \left\{ \Lambda_\gamma(c,H^g) + \frac{\rho}{c} \right\}
  \end{equation}
  for all $\Qb = \eta \otimes \nu$ such that
  $\RE(\eta \mid \gamma) \leq \rho$.
\end{corollary}

\Cref{cor:info-budget} suggests a bound across an infinite dimensional
family of alternative models,
\begin{equation*}
  \label{eq:alt-family}
  \mathcal{Q} := \{Q = \eta \otimes \nu 
  : \RE(\eta \mid \gamma) \leq \rho\}\,,
\end{equation*}
that includes both parametric and non-parametric perturbations as
depicted in \cref{fig:bounds-across-family}. We observe that the bound
\cref{eq:uq-bounds-rho} requires sampling the cumulant with respect to
the nominal model and then optimizing over $c$ once for a given $\rho$
and QoI. In contrast, computing the weak error
\cref{eq:weak-error-observables} directly for every
$\Qb \in \mathcal{Q}$ would be computationally infeasible. Although we
cannot give a general insight on how to interpret the information
budget, one can get a sense of the family of alternative models that
fall within a given budget, for example, how qualitatively different a
model you can buy for a fixed amount of relative entropy, by computing
a relative entropy landscape (cf.\
\cref{fig:mod1_relandscape,fig:experiment2_RE_hist}). In the context
of R{\'e}nyi divergence in a model of differential privacy,
\cite{DuchiJordanWainwright:2013lp} introduces a level of privacy,
$\alpha$, that is philosophically similar to $\rho$ when viewed as an
allowable level of model-form uncertainty.

\begin{figure}
  \centering
  \includegraphics[width=0.6\textwidth]{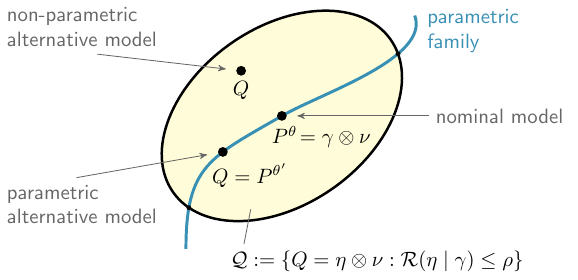}
  \caption{For a given information budget $\rho$,
    \cref{cor:info-budget} suggests tight and computable upper and
    lower bounds for the weak error \cref{eq:weak-error-observables}
    between a nominal model $\Pb$ for every alternative model
    $\Qb \in \mathcal{Q}$, an infinite dimensional family that
    includes both parametric and non-parametric perturbations. The
    information budget $\rho$ can be interpreted as an allowable level
    of model-form uncertainty.}
  \label{fig:bounds-across-family}
\end{figure}

\begin{remark}[Data Processing Inequality]
  \label{rmk:data-processing}
  For any invertible transformation $T$,
  \begin{equation*}
    \label{eq:re-data-processing-eq}
    \RE (\lambda \mid \gamma) = \RE (T(\lambda)\mid T(\gamma)),
  \end{equation*}
  from the Data Processing Inequality (see for example
  \cite{GilEtAl:2013rd}). Thus, we can replace the distribution of the
  conductivity $\bar{a}$ with the distribution of $\log \bar{a}$ in
  any relative entropy statement without a loss of information.
\end{remark}

\subsection{Uncertainty intervals}%
\label{sec:uncertainty-intervals}%

We end this section with another perspective on \cref{eq:xi-bound} for
data-informed prediction. For a given nominal model
$\Pb = \gamma \otimes \nu$, \cref{eq:xi-bound} can be rewritten as
\begin{equation*}
  \E_\Pb[g(\bar{u})] - \Xi(\lambda \mid \gamma; - H^g)
  \leq \E_\Qb[g(\bar{u})] \leq 
  \E_\Pb[g(\bar{u})] + \Xi(\lambda \mid \gamma; H^g),
\end{equation*}
yielding an uncertainty interval for an alternative model prediction.
The alternative model $Q = \lambda \otimes \nu$ falls within a region
of confidence given by the nominal model prediction $\pm \Xi$, i.e.\
$\E_\Qb[g(\bar{u})] \in \E_\Pb[g(\bar{u})] \pm \Xi(\lambda \mid
\gamma; \pm H^g)$,
providing a guarantee on a QoI with respect to an alternative model
$\Qb$. For failure probabilities \cref{eq:failure-probability} (such
as the goal functionals $g_1$ and $g_2$ to appear in
\cref{sec:screening-and-sa,sec:data-informed-bounds}), the hybrid
information divergences have a particularly simple form that gives a
confidence interval for the $\Qb$-probability of failure.

\begin{theorem}[Uncertainty Interval for $\Qb$-failure]
  \label{thm:uq-interval-failure}
  For a nominal model $\Pb = \gamma \otimes \nu$, let
  $g(\bar{u}) = \indic_A$ for $A \subset \Omega$ with $\Pb(A) = p$ and
  let $\rho := \RE(\lambda \mid \gamma)$. Then,
  \begin{equation*}
    - \min_{c>0} \left\{ \frac{1}{c}\log(pe^{-c}+1-p) 
      + \frac{\rho}{c}\right\} 
    \leq \Qb(A) \leq 
    \min_{c>0} \left\{ \frac{1}{c} \log(pe^c+1-p)
      + \frac{\rho}{c} \right\},
  \end{equation*}
  for every alternative model $\Qb = \eta \otimes \nu$ such that
  $\RE(\eta \mid \gamma) \leq \rho$.
\end{theorem}

\begin{proof}
  The risk-sensitive performance measure is
  \begin{equation}
    \label{eq:Lambda-failure-prob}
    \Lambda_\gamma(c, \pm H^g) 
    = \frac{1}{c} \log \E_\gamma [e^{\pm c(H^g - \E_\gamma[H^g])}] 
    =  \frac{1}{c} \log (pe^{\pm c} + 1 - p) \mp p
  \end{equation}
  and thus the bounds follows immediately from \cref{eq:xi-bound}.
\end{proof}

The remainder of this paper focuses on applications of the hybrid
information divergences to UQ. First, in \cref{sec:screening-and-sa}
we apply \cref{thm:information-divergence} to derive bounds for
parametric sensitivity analysis by considering alternative models that
can be represented by small parametric perturbations of the nominal
model. Then in \cref{sec:data-informed-bounds} we examine a more
exploratory UQ task and derive bounds for model misspecification due
to sparse data where the nominal and alternative models cannot, in
general, be described by small perturbations. Finally, in
\cref{sec:efficient-sampling-risk}, we leverage the connection between
certain concentration inequalities and the hybrid divergences for
efficient computing.

\section{The simplest UQ application: parametric sensitivity
  analysis}%
\label{sec:screening-and-sa}%

Presently we apply the tools developed in
\cref{sec:hybr-inform-diverg} to sensitivity analysis when the model
inputs are specified by a parametric geostatistical model. This
setting represents the simplest UQ application of the hybrid
information divergences in that we have tight control over the
perturbations and hence over the alternative models under
consideration (see \cref{fig:bounds-across-family} where the
parametric perturbations are tightly controlled). Although simple, the
example nonetheless represents an important UQ task and allows us to
demonstrate the tightness and robustness of the bounds derived from
the hybrid information divergences. In principle, this approach can be
applied to study non-parametric models, i.e.\ models that are infinite
dimensional in the parameter space, such as a gPC representation of
the conductivity in combination with a stochastic Galerkin method for
the solver. However, we note that for our specific application of
interest with a lognormal conductivity such an approximation is not
guaranteed to converge due to Proposition 4.2 in
\cite{ErnstEtAl:2012pc}.

In the section that follows, we begin by providing notation and
motivation for the parametric sensitivity analysis. In
\cref{sec:si-and-small-perturbations} we give
\cref{cor:screening-bound} containing a cheaply computed bound that
can be used to efficiently screen for insensitive parameter
directions. Then in \cref{sec:sa-reduced-model}, we apply
\cref{thm:information-divergence} to obtain accurate and robust bounds
for sensitivity analysis. Finally, in
\cref{sec:computability-of-bounds} we provide details on the
implementation. We emphasize that although a one-dimensional example
problem is considered, the techniques demonstrated easily scale to
higher dimensions.

\subsection{Parametric geostatistical models and sensitivity indices}%
\label{sec:motivation-sa}
For a given probability space $(\Omega, \mathcal{F}, \Pb)$ we consider
the two-point boundary value problem
\begin{equation}
  \label{eq:1d-rpde}
  -(a^\theta (\omega, x) u^\prime(\omega, x))^\prime = 1, 
  \qquad \text{for } x \in [0,1],
\end{equation}
subject to $u(\omega, 0) = 0$ and
$a^\theta(\omega, 1) u'(\omega, 1) = 1$, where randomness enters only
through a scalar-valued log-normal process $a^\theta$ that depends on
a vector of hyperparameters $\theta \in \mathbf{R}^k$. In particular,
we consider $\log a^\theta$ with mean $\mu$ and squared-exponential
type two-point covariance function $C$ given by,
\begin{equation}
  \label{eq:cov-se-nugget}
  C(r) =   \begin{cases} 
    \sigma^2 e^{-|r/\sqrt{2}\ell|^2}, 
    & \text{for } r > 0,\\
    \tau^2 + \sigma^2, & \text{for }  r = 0,
  \end{cases}
\end{equation}
where $r = |x - \tilde{x}|$ for $x,\tilde{x} \in [0,1]$.

For the mean and covariance above, $a^\theta$ is a stationary,
isotropic random field where the hyperparameters of interest are
$\theta = (\mu, \sigma^2, \ell, \tau^2) \in \mathbf{R}^4$. In
applications, these hyperparameters have geostatistical
interpretations that play a role in fitting the model for $a^\theta$
from data; $\mu$ is related to the overall trend, $\sigma^2$ is
related to the sill measurement, $\ell$ is related to the spatial
correlation length, and $\tau^2$ is related to the nugget effect or
microscale variability (\cite{GelfandEtAl:2010hb}). The sample paths
of the process $a^\theta$ exhibit qualitatively different behavior
across a range of hyperparameter values and it is therefore natural to
question the sensitivity of a QoI with respect to parametric modeling
assumptions on the conductivity field. With a view toward employing
the hybrid information divergences in \cref{sec:hybr-inform-diverg},
we denote the finite dimensional distributions
$(\bar{a}^\theta_n) \sim \gamma$ and
$(\bar{a}^{\theta'}_n) \sim \gamma'$ where
$\theta' = \theta + \epsilon v$ is a small perturbation for
$\epsilon >0$ in the direction $v \in \mathbf{R}^4$ with $|v|=1$. Then
we consider the joint probability measures
$\Pb^\theta = \gamma \otimes \nu$ and
$\Pb^{\theta'} = \gamma' \otimes \nu$ that correspond to the nominal
and perturbed parameters of the geostatistical model where we denote
the distribution of the corresponding finite element solution by
$(\bar{u}_n) \sim \nu(\dd z \mid \cdot)$. We would like to understand
the sensitivity of $\E_{\Pb^\theta}[g(\bar{u})]$ with respect to
distributional assumptions on $\Pb^\theta$ and in particular to
quantify worst-case scenarios concerning this sensitivity with a view
toward informing decision tasks.

For a given goal functional $g$, we define the sensitivity index,
\begin{equation}
  \label{eq:parametric-si}
  \mathcal{S} (v, \theta; g) 
  = \theta \lim_{\epsilon \to 0} \frac{
    \E_{\Pb^{\theta+\epsilon v}}[g(\bar{u})] 
    - \E_{\Pb^{\theta}}[g(\bar{u})]}{\epsilon},
\end{equation}
that describes the sensitivity of a given goal functional $g$ with
respect to $\theta$ in the direction $v$, provided $\mathcal{S}$
depends continuously on $\theta$. In the limit of small $\epsilon$,
$\mathcal{S}$ converges to the logarithmic derivative
$\partial_{\log \theta} \E_{\Pb^\theta}[g(\bar{u})] = \theta
\partial_\theta \E_{\Pb^\theta}[g(\bar{u})]$, a scaling chosen to
control for differences in the orders of magnitude of the
hyperparameters.

Computing a classical gradient approximation of $\mathcal{S}$ in each
parameter direction for each QoI represents a nontrivial computational
cost even for the simple model problem \cref{eq:1d-rpde}. A naive
finite difference approximation of the sensitivity index would require
sampling with respect to both $\Pb^\theta$ and $\Pb^{\theta^\prime}$
where each sample involves a call to a PDE solver for each direction
$v$ in $\theta' = \theta + \epsilon v$. Moreover, such a gradient
approximation introduces a bias error that must be taken into account;
for a better approximation of the sensitivity, corresponding to small
$\epsilon$, the variance of the approximation increases and therefore
our confidence of it decreases. While reduced variance methods for
gradient approximations exist (\cite{Glasserman:2003mc,
  GlassermanYao:1992gg}), our direction here is an altogether
different one. In contrast, \cref{eq:xi-bound} in
\cref{thm:information-divergence} yields tight non-gradient based
estimates for the sensitivity \cref{eq:parametric-si} that only
require sampling with respect to the nominal model $\Pb^\theta$.
Before considering these more accurate bounds in
\cref{sec:sa-reduced-model}, we first demonstrate in
\cref{sec:si-and-small-perturbations} a cheaply computed bound,
derived from expansion \cref{eq:xi-linearization} in
\cref{thm:information-divergence}, that can be used to screen for
insensitive parameter directions.

\subsection{Fast screening for small perturbations}%
\label{sec:si-and-small-perturbations}%

While it is always advantageous if the number of parameters to include
in the full sensitivity analysis can be reduced, the efficiency gain
is acute when simulations are computationally expensive such as
involving successive calls to a PDE solver. Following from
\cref{eq:xi-linearization} in \cref{thm:information-divergence}, we
consider a linearization specialized to small perturbations that
relies on the Fisher Information matrix (FIM). As the FIM and can be
computed cheaply, that is, without sampling, the linearized bound can
be used to efficiently screen for insensitive parameter directions. We
recall that the FIM for a parametric family of distributions
$\Pb^\theta$ is given by
\begin{equation*}
  \label{eq:fim}
  \mathcal{I}(\theta) := \int_{\rset^d} \nabla_\theta \log p(x;\theta) 
  (\nabla_\theta \log p(x;\theta))^\top p(x;\theta) \dd x,
\end{equation*}
where $p(x; \theta)$ is the density conditional on the value of
$\theta$ (a classical definition from \cite{Wasserman:2013as}).

\begin{corollary}[Efficient Screening]
  \label{cor:screening-bound}
  For a smooth parametric family $\Pb^\theta$ and $\epsilon >0$,
  \begin{equation*}
    \label{eq:screening-bound}
    \frac{1}{\epsilon} |\E_{\Pb^{\theta'}}[g(\bar{u})] 
    - \E_{\Pb^{\theta}}[g(\bar{u})] |
    \le \sqrt{\var_{\gamma}[H^g]} 
    \sqrt{v^\top \mathcal{I}(\theta) v} + O(\epsilon)
  \end{equation*}
  where $\mathcal{I}(\theta)$ is the FIM associated with $\Pb^\theta$.
  Hence
  \begin{equation}
    \label{eq:S-screening-bound}
    |\mathcal{S}(v,\theta;g)| \leq \theta \sqrt{\var_{\gamma}[H^g]}
    \sqrt{v^\top \mathcal{I}(\theta) v}.
  \end{equation}
\end{corollary}

The bounds suggested in \Cref{cor:screening-bound} are for small
$\epsilon$ perturbations of the nominal model in terms of the variance
of the marginal performance measure $H^g$ and the FIM. Recall that
$H^g$ is related to the distribution of the QoI given the distribution
of input parameters. Reminiscent of the variance-bias trade-off of
other information information-based criterion for assessing model
selection in statistics (for example Akaike Information Criterion and
Bayesian Information Criterion \cite{BurnhamAnderson:2002ms}),
importantly \cref{eq:S-screening-bound} is a goal-oriented quantity
that incorporates the output of the forward model $H^g$.

\Cref{cor:screening-bound} follows from the general non-infinitesimal
linearization \cref{eq:xi-linearization} by noticing that the relative
entropy has the expansion
\begin{equation*}
  \RE(\Pb^{\theta + \epsilon v} \mid \Pb^\theta) 
  = \frac{\epsilon^2}{2} v^\top \mathcal{I}(\theta)v + O(\epsilon^3)
\end{equation*}
when considering small perturbations to a smooth parametric family of
probability measures; the complete proof follows from results in
\cite{DupuisEtAl:2015ps}. A similar linearization has been used in
chemical kinetics to screen for insensitive parameter directions in
the situation where the number of parameters is large
(\cite{ArampatzisEtAl:2015as, TsourtisEtAl:2015}). For both a
multivariate normal distribution and log-normal distribution
characterized by mean $\mu(\theta)$ and covariance $\Sigma(\theta)$,
the $i,j$ component of FIM can be expressed as
\begin{equation}
  \label{eq:fim-log-norm}
  v_i^\top \mathcal{I}(\theta) v_j
  = \frac{\partial \mu^\top}{\partial \theta_i} \Sigma 
  \frac{\partial \mu}{\partial \theta_j} 
  + \frac{1}{2} \trace\left(\Sigma^{-1} 
    \frac{\partial \Sigma}{\partial \theta_i} 
    \Sigma^{-1} \frac{\partial \Sigma}{\partial \theta_j} \right),
\end{equation}
(see for example \cite{KomorowskiEtAl:2011sa}). An analysis of the
singular value decomposition of the FIM can then reveal arbitrary
parameter directions that are relatively insensitive to perturbations.

Recall that $\mathcal{S}$, defined in \cref{eq:parametric-si}, is the
sensitivity index based on the logarithmic derivative to control for
differences in magnitude between the parameters. At present we assume
that choice of solver is independent of the choice of parameters; in
this case, the variance term of the marginal performance measure in
\cref{eq:S-screening-bound} is fixed and the sensitivity of parameters
can be screened based on the appropriately scaled FIM appearing in
\cref{eq:S-screening-bound}. In general, the solver and hence $\nu$
might depend on the parameters as these affect the regularity of the
random field. In \cref{fig:fim_SE_screening}, the screening index,
\begin{equation}
  \label{eq:screening-index}
  J(i,i) = \theta_i \sqrt{ v_i^\top \mathcal{I}(\theta) v_i },
\end{equation}
is calculated for each $i = 1,\dots,4$, that is for each of the
principal parameter directions $\{\mu, \sigma^2, \ell, \tau^2\}$
corresponding to the diagonals of the FIM. These indices are compared
across a range of nominal models for a fixed goal functional.
\Cref{fig:fim_SE_screening} demonstrates the relative insensitivity of
perturbations in $\mu$ and $\sigma^2$ over the range of nominal models
where $\ell$ and $\tau^2$ vary for $\mu = 0.8$ and $\sigma^2 = 4$. $J$
is computed without sampling using expression \cref{eq:fim-log-norm}
and identifies that the directions $\mu$ and $\sigma^2$ might be
excluded from the full sensitivity analysis for the given goal
functional since the screening index is small relative to the value
for other directions.

\begin{figure}
  \centering
  \includegraphics[scale=1]{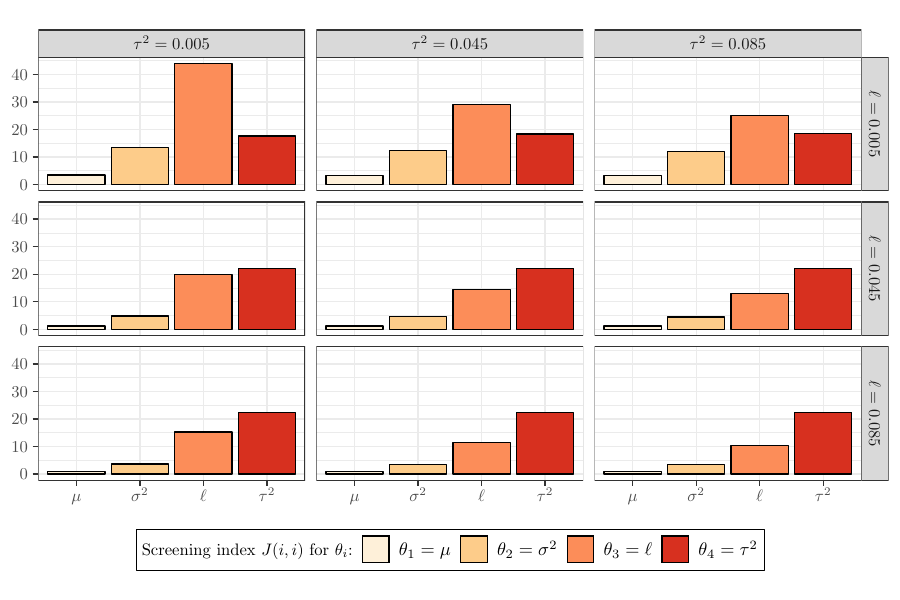}
  \caption{The screening index $J$ in \cref{eq:screening-index} for
    the logarithmic-derivative based parametric sensitivity depends on
    the FIM and can be computed cheaply (without sampling) thus
    providing an efficient method for screening parameters. In this
    instance, $J$ indicates the directions $\mu$ and $\sigma^2$ are
    relatively insensitive compared to perturbations in $\ell$ and
    $\tau^2$ over nine different nominal models for a fixed QoI
    corresponding to a failure probability.}
  \label{fig:fim_SE_screening}
\end{figure}

\subsection{Robust bounds and worst-case scenarios}%
\label{sec:sa-reduced-model}%

Next, we demonstrate bounds based on \cref{eq:xi-bound} in
\cref{thm:information-divergence} that are more accurate than the
linearized bounds at the cost of being more computationally expensive.
To investigate the performance for parametric sensitivity analysis, we
fix a nominal model $P^\theta$, with hyperparameters
$\theta = (\mu=0.8, \sigma^2=4, \ell=0.005, \tau^2=0.045)$, and
consider the sensitivity with respect to alternative models
$P^{\theta+\epsilon v}$ corresponding to small perturbations in the
$\ell$ and $\tau^2$ which we denote by $\epsilon(\ell)$ and
$\epsilon(\tau^2)$ (see also \cref{fig:mod1_relandscape}). We also fix
the goal functionals
\begin{subequations}
  \begin{align}
    &g_1(\bar{u}) = \indic_{\{\bar{u}(1) > 1.2\}}\,, 
      \label{eq:sa-goal-functionals-g1}\\
    &g_2(\bar{u}) =\indic_{\{0.25 < \bar{u}(1) <0.75\}}\,, 
      \quad \text{and} \label{eq:sa-goal-functionals-g2}\\
    &g_3(\bar{u}) = \min (u(1), 3)\,. 
      \label{eq:sa-goal-functionals-g3}
  \end{align}
\end{subequations}
\Cref{eq:sa-goal-functionals-g1,eq:sa-goal-functionals-g2} are
indicator functions (QoI corresponding to failure probabilities cf.\
\cref{thm:uq-interval-failure}) and \cref{eq:sa-goal-functionals-g3}
is a point estimate with an enforced upper bound.

In \cref{fig:mod1_uq_bounds_ell,fig:mod1_uq_bounds_tau}, a scaled
hybrid information divergence \cref{eq:xi}
\begin{equation}
  \label{eq:xi-pm-star}
  \frac{\theta}{\epsilon} \, \Xi_{\pm} (\gamma' \mid \gamma; H^{g_i}),
\end{equation}
is compared to the reference quantity,
\begin{equation}
  \label{eq:finite-diff}
  \hat{\Delta}(\epsilon, M; g_i) = 
  \frac{\theta}{\epsilon} \left(E_{\gamma'}^M[g_i] 
    - E_{\gamma}^M [g_i]\right),
\end{equation}
a finite difference approximation of the sensitivity $\mathcal{S}$
where
\begin{equation}
  \label{eq:M-sample-avg} E_\gamma^M (f) = \frac{1}{M} \sum_{j=1}^{M}
  f(\omega_j)
\end{equation} 
denotes the sample average based on $M$ independent and identically
distributed samples of $f$ drawn with respect to $\gamma$. Each
observation appearing in
\cref{fig:mod1_uq_bounds_ell,fig:mod1_uq_bounds_tau} is based on the
mean of $\num{e2}$ runs of $M=\num{e3}$ samples and the confidence
intervals denote two standard deviations from the corresponding sample
mean.

\begin{figure}
  \centering \subfloat[]{\label{fig:mod1_uq_bounds_ell}
    \includegraphics[scale=0.9]{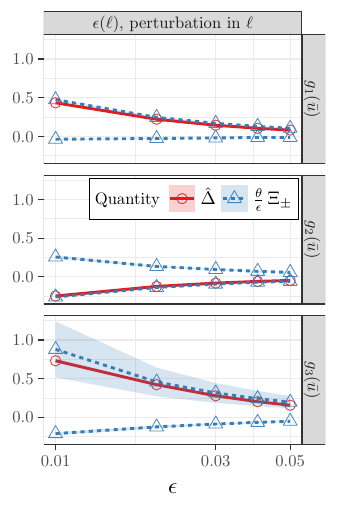}}
  \subfloat[]{\label{fig:mod1_uq_bounds_tau}
    \includegraphics[scale=0.905]{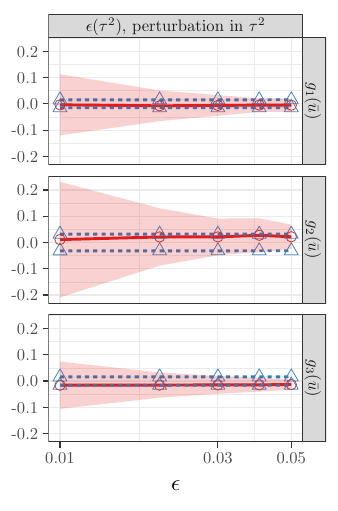}} \hfill
  \subfloat[]{\label{fig:mod1_relandscape}
    \includegraphics[scale=.77]{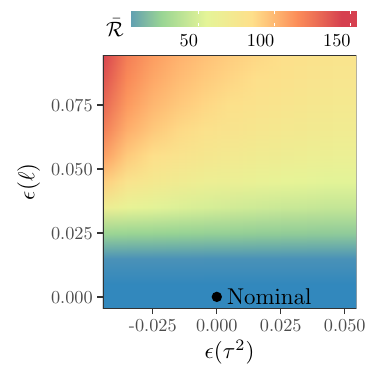}}
  \caption{In (a) and (b), the bounds \cref{eq:xi-pm-star} provide a
    tight estimate of the logarithmic-derivative based parametric
    sensitivity index \cref{eq:finite-diff} for bounded goal
    functionals
    \cref{eq:sa-goal-functionals-g1,eq:sa-goal-functionals-g2,eq:sa-goal-functionals-g3}.
    The shaded regions in \cref{fig:mod1_uq_bounds_ell} and
    \cref{fig:mod1_uq_bounds_tau} denote confidence intervals of two
    standard deviations from the corresponding sample mean
    \cref{eq:M-sample-avg} based on $10^2$ runs of $M=10^3$ samples.
    Comparing (a) to (b) we observe that the bounds in (b) are
    captured by the corresponding bounds in (a); the relative entropy
    landscape in (c) reveals an explanation, namely, that the
    information budget established by $\epsilon(\ell)$, i.e.\ small
    $\epsilon$ perturbations in $\ell$, contains corresponding
    perturbations $\epsilon(\tau^2)$. Thus, we note the bounds
    \cref{eq:xi-pm-star} in (a) are robust, providing a worst-case
    scenario envelope for all perturbations within an information
    budget (cf.\ \cref{cor:info-budget}).}
  \label{fig:mod1}
\end{figure}

We observe in \cref{fig:mod1_uq_bounds_ell,fig:mod1_uq_bounds_tau}
that for each QoI (facet corresponding to rows), the bounds
\cref{eq:xi-pm-star} provide an accurate estimate of the sensitivity
for $\epsilon(\ell)$ and $\epsilon(\tau^2)$. In particular, the plots
are suggestive of the tightness of the bounds derived from the hybrid
information divergence. Further, a comparison of
\cref{fig:mod1_uq_bounds_ell} to \cref{fig:mod1_uq_bounds_tau},
indicates that the bounds \cref{eq:xi-pm-star} are robust. For this
particular nominal model, the relative entropy landscape in
\cref{fig:mod1_relandscape} shows that $\epsilon(\tau^2)$ always fall
within the information budget established by $\epsilon(\ell)$, that
is, the level sets relating to $\epsilon(\ell)$ contain the
corresponding $\epsilon(\tau^2)$. Thus, the bounds
\cref{eq:xi-pm-star} in \cref{fig:mod1_uq_bounds_ell} are guaranteed
to contain the bounds in \cref{fig:mod1_uq_bounds_tau} by
\cref{cor:info-budget} and can be interpreted as giving the worst-case
scenario for each QoI, providing a natural way to rigorously
incorporate worst-case scenarios into the decision support framework
in \cref{fig:layers-subsurface-flow-model}.

\subsection{\emph{A posteriori} computability}%
\label{sec:computability-of-bounds}%

We emphasize that the components appearing in the argument of the
optimization problem in \cref{eq:xi}, and in particular in the
right-hand side of \cref{eq:xi-pm-star}, are \emph{a posteriori}
computable quantities that represent significant computational savings
over gradient approximations. Moreover, \cref{eq:xi-pm-star}
incorporate worst-case scenarios that might not be efficiently
observed using traditional estimates of the sensitivity.
In the previous numerical experiment, we approximate
\cref{eq:xi-pm-star} by
\begin{equation*}
  \Xi_{\pm} (\gamma' \mid \gamma ; H^g) \approx \pm \xi(c^*, \pm H^g)
\end{equation*}
where $c^* = \argmin_{c>0} \xi (c, H^g)$ and
\begin{equation}
  \label{eq:estimate-xi}
  \xi (c , H^g) 
  := \hat{\Lambda}_\gamma(c, H^g) 
  + \frac{1}{c} \bar{\RE}(\gamma' \mid \gamma)
\end{equation} 
for suitable approximations $\hat{\Lambda}_\gamma$ and $\bar{\RE}$ of
the risk-sensitive performance measure and relative entropy,
respectively. The optimal $c^*$ as a function of
$\rho := \RE(\lambda \mid \gamma)$ has the representation
\begin{equation}
  \label{eq:optimal-cstar-rho} 
  c^*(\rho) = \frac{\sqrt{2
      \rho}}{\sqrt{\var_\gamma[H^g]}} + O(\rho),
\end{equation}
which follows from equation (2.28) of \cite{DupuisEtAl:2015ps}. For
perturbations resulting in small $\rho$ \cref{eq:optimal-cstar-rho}
can be used; otherwise, we find the optimal $c^*$ by a one-dimensional
Newton-Raphson method, a step that must be repeated for each QoI for
every alternative model under consideration.

The $\hat{\Lambda}_\gamma$ appearing in \cref{eq:estimate-xi} can be
sampled using a standard MC approximation,
\begin{equation*}
  \hat{\Lambda}_\gamma(c,H^g) =
  \frac{1}{c} \log E_\gamma^M (\exp\{c (H^g - \widehat{H^g})\}) 
  \approx \Lambda_\gamma(c,H^g),
\end{equation*}
for $\widehat{H^g} := E_\gamma^M(H^g) \approx \E_\gamma[H^g]$ where
$E_\gamma^M$ denotes the sample average \cref{eq:M-sample-avg}. This
quantity needs to be computed only once for each QoI, according to the
nominal model $\gamma$, and can then be used as in
\cref{eq:uq-bounds-rho} to test any number of alternative models
within the established information budget as in
\cref{cor:info-budget}. In contrast, the relative entropy appearing in
\cref{eq:estimate-xi} needs to be computed for every alternative model
under consideration. However, the relative entropy can be computed
without sampling using the analytic formula \cref{eq:re-multivar-norm}
together with \cref{rmk:data-processing} to replace the distribution
of the conductivity with the corresponding Gaussian. In the preceding
experiments the approximation
$\bar{\RE}(\gamma'\mid\gamma) \approx \RE(\gamma'\mid\gamma)$ is
obtained by taking the Gaussians to have the same dimension as the
finite element discretization.

\begin{remark}[Cumulant generating functions amplify variance]
  \label{rmk:sampling-pm}
  Forming an estimator involving the hybrid information divergence
  $\Xi$ as in \cref{eq:xi-pm-star} requires sampling a cumulant
  generating functional ($\hat{\Lambda}_\gamma$). Although sampled
  with respect to the nominal model, if the variance of the marginal
  performance measure is large then this variance could potentially be
  amplified by $\hat{\Lambda}_\gamma$; in
  \cref{fig:mod1_uq_bounds_ell}, the estimator of $\Xi_{+}$ for goal
  functional $g_3$ \cref{eq:sa-goal-functionals-g3} is observed to
  have higher variance than the indicator functionals
  \cref{eq:sa-goal-functionals-g1,eq:sa-goal-functionals-g2}. In
  \cref{sec:efficient-sampling-risk}, we demonstrate concentration
  inequality bounds for $\Lambda_\gamma$ that results in reduced
  variance predictions the information divergences $\Xi_\pm$.
\end{remark}

\begin{remark}[Cholesky-like covariance decomposition]
  The formula \cref{eq:re-multivar-norm} depends on the discrete
  projection $(\bar{a}_n) \approx a$. In some instances $\Sigma$ may
  be close to singular, hampering the computation of the precision
  matrix $\Sigma^{-1}$ or the $\log$-determinant. In the numerical
  experiments presented here, such issues were easily addressed using
  a Cholesky-like covariance decomposition and facts about Toeplitz
  matrices. Geostatistical models based on Markov random fields
  (\cite{RueHeld:2005}), as opposed to parametric covariance models,
  is an approach that sidesteps this difficulty and we note the
  techniques outlined here also apply to conductivities given by
  Markov random fields (see \cite{GourgouliasEtAl:2017aa}).
\end{remark}

In the next section, we examine non-parametric perturbations to a
geostatistical model. In particular, \cref{thm:information-divergence}
yields tight and robust UQ bounds in the context of model
misspecification due to sparse data.

\section{Data-informed error bounds for non-parametric perturbations}%
\label{sec:data-informed-bounds}%

In the present section, we consider model-form uncertainty in
connection with misspecification of the geostatistical model due to
lacking or incomplete data. As emphasized in the introduction, data
for our applications of interest are sparse and small perturbations to
data result in geostatistical models that are non-parametric
perturbations of the nominal model (see
\cref{fig:bounds-across-family}). By allowing us to compare the effect
that distributional assumptions on model inputs have on model outputs,
the hybrid information divergence provides a link between data and
decision tasks. In this vein, we explore how the hybrid information
divergences complement an existing inference procedure by providing
robust, data-informed bounds that give a sense of worst-case scenarios
under modeling errors. Next, we review the data set and the inference
procedure used in our experiments.

\subsection{Conductivity data and model problem}%
\label{sec:cond-data-model-prob}%

We utilize permeability data for a Brent sequence (\SI{365.76}{\meter}
by \SI{670.56}{\meter} by \SI{51.816}{\meter}) from SPE10 model 2 in
\cite{ChristieBlunt:2001sp}. Due to its importance in the North Sea
petroleum industry, the Brent sequence is well studied from a
geological perspective (\cite{Richards:1992bg}). The sequence has two
distinct phases; the upper layers of the sequence comprise a Tarbert
formation and the bottom layers comprise an Upper Ness formation. The
log-permeability of these two formations both vary by several orders
of magnitude and exhibit strikingly different spatial correlations. In
our numerical experiments, we will fit various geostatistical models
based on data from a one-dimensional slice of the upper-most level of
the Tarbert formation displayed in \cref{fig:conductivity_slice}.

\begin{figure}[]
  \centering \subfloat[]{\label{fig:conductivity_slice}
    \includegraphics[scale=0.76]{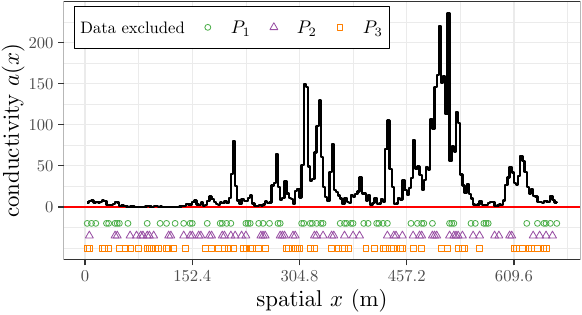}} \hfill
  \subfloat[]{\label{fig:experiment2_RE_hist}
    \includegraphics[scale=0.79]{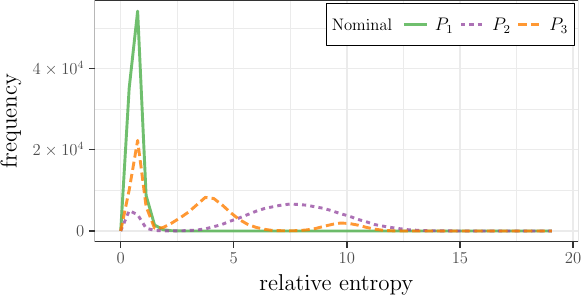}}
  \caption{In (a), a one-dimensional slice of the Tarbert formation
    data (from SPE10 model 2 in \cite{ChristieBlunt:2001sp}) used in
    numerical experiments in \cref{sec:data-informed-bounds} varies by
    orders of magnitude over the problem scale. In (b), three
    frequency distributions for the relative entropy are depicted for
    non-small perturbations of three different nominal models
    corresponding to $P_1$, $P_2$ and $P_3$ in (a). We observe that
    relative entropy distributions with respect to the family of
    alternative models \cref{eq:q-plus} are multi-modal and depend in
    a nontrivial fashion on the nominal model (cf.\ the relative
    entropy landscape in \cref{fig:mod1_relandscape} where the small
    parametric perturbations result in a relatively flat landscape in
    the $\epsilon(\tau^2)$ direction with a smooth ascent in
    $\epsilon(\ell)$).}
  \label{fig:experiment2}
\end{figure}

For the experiments that follow, we fix both a parametric form for the
geostatistical model and an inference procedure. We then consider
different geostatistical models fit from incomplete samples of the
full data set. Specifically, we assume that the available log-data are
Gaussian and then fit the parameters of the geostatistical model using
the maximum likelihood method. For a given parametric model, this
method gives parameter values that are found to maximize the
likelihood of making an observation of a particular data point given
the parameter value; this process is entirely automated by a number of
software packages and the present experiments use
`\texttt{RandomFields}' (\cite{SchlatherEtAl:2015rf}) available in R.
For convenience we shall again use the covariance
\cref{eq:cov-se-nugget} from \cref{sec:screening-and-sa}. Although we
posit a parametric form, geostatistical models resulting from fits
relying on incomplete observations of the full data set are not in
general small parametric perturbations of one another. Even small
changes to these discrete degrees of freedom may result in global
changes to parameters and hyperparameters, in contrast to the
localized sensitivity analysis in \cref{sec:screening-and-sa}.

We again consider the one-dimensional model problem \cref{eq:1d-rpde},
where the conductivity fields $(\bar{a}_n)$ are generated on a regular
uniform mesh of $n$ equally spaced cells and this projection is used
in forming the stiffness matrix for the FEM computation as well as the
covariance matrices required for the relative entropy calculations
(i.e.\ $n=d$). The FEM solution $(\bar{u}_{2n})$ is then computed
using standard, piecewise linear elements on a coarse mesh with
diameter $2n$.

In the remainder of the present section, we describe two numerical
experiments that use the hybrid information divergence
\cref{eq:xi-bound} to obtain data-informed bounds. The first
experiment in \cref{sec:model-misspecification} provides a sense of
the modeling error due to misspecification stemming from incomplete
data over a range of changes to the discrete degrees of freedom. In
the second experiment in \cref{sec:finding-worst-case-scenarios}, we
fix a nominal model based on a portion of the full data set and
examine the distribution of the relative entropy to identify an
information budget such that the hybrid information divergences give
robust bounds that include worst-case scenarios.

\subsection{Hybrid information divergences for model
  misspecification}%
\label{sec:model-misspecification}%

In \cref{fig:experiment1_g12,fig:experiment1_g3}, we demonstrate the
sensitivity of the modeling error \cref{eq:weak-error-observables}
with respect to changes in the geostatistical model resulting from the
inclusion or exclusion of a small number of data points. This
sensitivity is with respect to discrete changes to the degrees of
freedom used to fit the parametric model and is not understood in the
same sense as \cref{eq:parametric-si}. The weak errors in
\cref{fig:experiment1_g12,fig:experiment1_g3} are between a nominal
model, based on a portion of the available data set, and alternative
models that correspond to including or excluding a fixed number of
points from the data used to construct the nominal model. The
data-informed bounds derived from the hybrid information divergence
give tight and robust predictions for this weak error.

\begin{figure}[]
  \centering
  \includegraphics[scale=1]{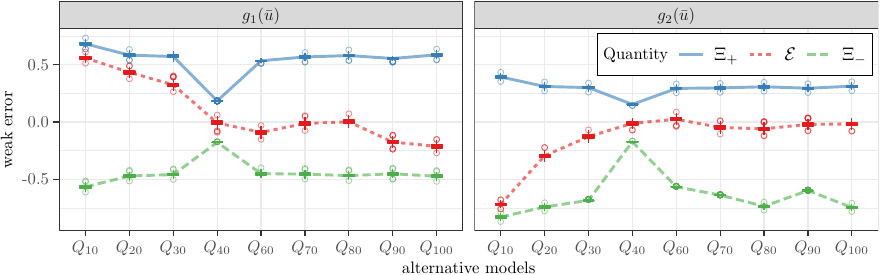}
  \caption{For the failure probability goal functionals
    \cref{eq:data-goal-functionals-g1,eq:data-goal-functionals-g2},
    the data-informed bounds $\Xi_\pm$ \cref{eq:data-xi-pm-star} give
    a tight and robust prediction of the weak error $\mathcal{E}$
    \cref{eq:weak-error-observables} over a range of changes to the
    discrete degrees of freedom. In (a) and (b) above, box plots are
    given for $\num{e2}$ observations, each of $M = \num{e3}$, samples
    where trend lines through the mean are added to indicate how the
    bounds form an envelope around model predictions.}
  \label{fig:experiment1_g12}
\end{figure}

\begin{figure}[]
  \centering
  \includegraphics[scale=1]{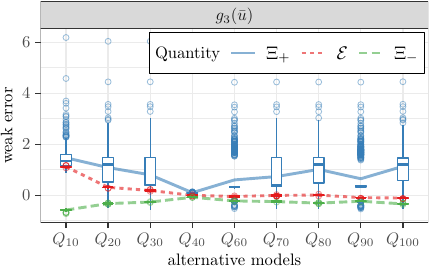}
  \caption{The data-informed bounds $\Xi_{\pm}$
    \cref{eq:data-xi-pm-star} provide a tight and robust estimate of
    the weak error $\mathcal{E}$ \cref{eq:weak-error-observables} for
    an unbounded $g_3$ \cref{eq:data-goal-functionals-g3}, however
    $\Xi_{+}$ have high variance (cf.\ \cref{rmk:sampling-pm}). We
    emphasize that the predictions here and in
    \cref{fig:experiment1_g12} are robust in that $\Xi_{\pm}$ bound
    the weak error for all alternative models that fall within a given
    information budget thus including a sense of worst-case scenarios.
    Above, box plots are given for $\num{e2}$ observations, each of
    $M = \num{e3}$, samples where trend lines through the mean are
    added to indicate how the bounds form an envelope around model
    predictions.}
  \label{fig:experiment1_g3}
\end{figure}

In particular, we begin by fixing a data set for the nominal model
$\gamma$ by sampling, uniformly at random, $50$ percent of the full
data set depicted in \cref{fig:conductivity_slice}. We then fit a
squared exponential covariance model \cref{eq:cov-se-nugget} using the
maximum likelihood method. We also fit a collection of alternative
models $\{\lambda_{10}, \dots, \lambda_{100}\}$ where $\lambda_q$ is
related to a geostatistical model that is fit using $q$ percent of the
full data set where a small number of points are added or deleted from
the subset of data used for neighboring alternative models. For
example, the data set used to construct $\lambda_{60}$ is formed by
sampling $10$ percent of the data points from the full data set not
included in $\lambda_{50}$ and then adding them to the partial set
used for $\lambda_{50}$. In keeping with the notation used in previous
sections, we then denote the nominal product measure
$\Pb = \gamma \otimes \nu$ and the alternatives
$\Qb_q = \lambda_q \otimes \nu$. Thus, the weak error displayed in
\cref{fig:experiment1_g12,fig:experiment1_g3} correspond to
alternative models related to a perturbation of the observed data set
used to fit the models, that is, to a perturbation of discrete degrees
of freedom. 

For this collection of nominal and alternative geostatistical models,
the bounds
\begin{equation}
  \label{eq:data-xi-pm-star}
  \Xi_{\pm} (\lambda_q \mid \gamma; H^g) \approx \pm \xi (c^*, \pm H^g)
\end{equation}
are computed using an expression similar to \cref{eq:estimate-xi} in
the spirit of \cref{sec:computability-of-bounds}. Box plots for
\num{e2} observations of each bound and weak error, each based on
$M = \num{e3}$ samples, are displayed in
\cref{fig:experiment1_g12,fig:experiment1_g3} along with a trend line
corresponding to the mean of the observations. The bounds and the weak
errors are examined for the goal functionals,
\begin{subequations}
  \begin{align}
    &g_1(\bar{u}) = \indic_{\{\bar{u}(x_1) > m\}}\,, 
      \label{eq:data-goal-functionals-g1}\\
    &g_2(\bar{u}) = \indic_{\{ m+s > \bar{u}(x_1) > m-s\}}\,, 
      \quad \text{and} \label{eq:data-goal-functionals-g2}\\
    &g_3(\bar{u}) = \bar{u}(x_1) / m \,, 
      \label{eq:data-goal-functionals-g3}
  \end{align}
\end{subequations}
where $m$ is the sample average of $\bar{u}(x_1)$ at the right-hand
endpoint of the domain and $s$ is the corresponding standard
deviation. We note that for
\cref{eq:data-goal-functionals-g1,eq:data-goal-functionals-g2} the
weak error is bounded in $[-1,1]$ whereas for
\cref{eq:data-goal-functionals-g3} it is unbounded and we therefore
anticipate any estimate related to \cref{eq:data-goal-functionals-g3}
to naturally have higher variance.

In \cref{fig:experiment1_g12}, related to
\cref{eq:data-goal-functionals-g1,eq:data-goal-functionals-g2}, we
observe that there is a fairly wide spread in the values for the weak
errors corresponding to different alternative models. In this
instance, the data-informed bounds \cref{eq:data-xi-pm-star} form a
tight envelope around this spread. Even in the case of
\cref{eq:data-goal-functionals-g3}, \cref{fig:experiment1_g3}
illustrates that \cref{eq:data-xi-pm-star} gives a reliable estimate
of the weak error. However, we observe that the estimator for $\Xi_+$
has high variance in this instance due to sampling strategy used for
the risk-sensitive performance measure $\hat{\Lambda}_\gamma$. As
noted in as in \cref{rmk:sampling-pm}, an alternative method will be
discussed in \cref{sec:efficient-sampling-risk}.

\subsection{Finding worst-case scenarios related to incomplete data}
\label{sec:finding-worst-case-scenarios}

Presently we examine worst-case scenarios due to changes in the
discrete degrees of freedom. The distribution of the relative entropy
with respect to a training set is used to determine an information
budget that yields robust, data-informed bounds encapsulating
worst-case scenarios.

\begin{figure}[]
  \centering
  \includegraphics[scale=1]{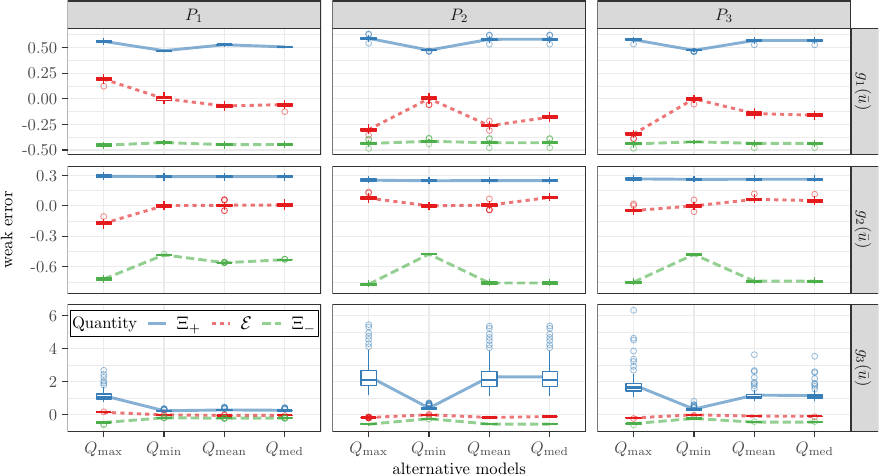}
  \caption{Choosing an alternative model $\Qb_{\max}$ related to the
    maximum relative entropy observed in a training set yields robust,
    data-informed bounds $\Xi_{\pm}$ that include a sense of
    worst-case scenarios related to the impact of incomplete data on
    the modeling process for the nominal models $\Pb_1$, $\Pb_2$, and
    $\Pb_3$ in \cref{fig:experiment2}. The matrix of plots above
    contains box plots of $\num{e2}$ observations, each of
    $M = \num{e3}$ samples, where trend lines through the means have
    been added to indicate how the bounds form an envelope around
    model predictions.}
  \label{fig:experiment2_bounds}
\end{figure}

We consider three different nominal models,
$\Pb_1 = \gamma_1 \otimes \nu$, $\Pb_2 = \gamma_2 \otimes \nu$, and
$\Pb_3 = \gamma_3 \otimes \nu$, that are related to fitting a
parametric geostatistical model $\gamma_i$ to $70$ percent of the full
data set sampled uniformly. In \cref{fig:conductivity_slice}, the full
data set is displayed in addition to the corresponding ``gaps'' in the
three different nominal models. As in previous sections, these models
are fit to a squared exponential covariance model
\cref{eq:cov-se-nugget} using the maximum likelihood method. For each
nominal model, we build the training set
\begin{equation}
  \label{eq:q-plus}
  \mathcal{Q}_i = \{ \Qb^+ = \lambda^+ \otimes \nu \}\, ,
  \quad i=1, 2, 3\, ,
\end{equation}
a collection of alternative models based on $\lambda^+$ that are fit
to enlargements of the nominal model data (i.e.\ to $80$ percent of
the full data set) where points are added by sampling those excluded
from the nominal model set uniformly. The corresponding frequency
distribution of the relative entropy for each $\Pb_i$ with respect to
$|\mathcal{Q}_i| = \num{e5}$ alternative models is displayed in
\cref{fig:experiment2_RE_hist}. The distributions, which exhibit
multi-modality, demonstrate the non-trivial dependence of the relative
entropy on the nominal model under consideration. These frequency
distributions play a similar role to the relative entropy landscape in
\cref{fig:mod1_relandscape} since in this case there is no natural
ordering among the alternative models (cf.\
\cref{fig:mod1_relandscape} where the small epsilon perturbation
provides a basis for ordering the models). The tightness of the hybrid
information divergence suggests that the modeling error
\cref{eq:weak-error-observables} may cluster according to the peaks in
the relative entropy distribution. Further, the information budget
established by the maximum observed relative entropy can be used to
bound all the alternative models in $\mathcal{Q}_i$ according to
\cref{cor:info-budget}.

From each $\mathcal{Q}_i$, we select four alternative models
$\Qb_{\max} = \lambda_{\max} \otimes \nu$,
$\Qb_{\min} = \lambda_{\min} \otimes \nu$,
$\Qb_{\mathrm{mean}} = \lambda_{\mathrm{mean}} \otimes \nu$, and
$\Qb_{\mathrm{med}} = \lambda_{\mathrm{med}} \otimes \nu$ that
correspond to the maximum, minimum, mean, and median relative entropy
with respect to $\Pb_i$, respectively, see also
\cref{fig:experiment2_RE_hist}. In \cref{fig:experiment2_bounds}, we
observe that once again that the bounds $\Xi_{\pm}$ yield robust
predictions for the modeling error between the nominal and each of
the alternative models. As expected, we observe that the weak error
corresponding to $\Qb_{\max}$ appear to be worst-case scenarios and
that this error is reliably contained in the envelope defined by
$\Xi_{\pm}$. In the present setting, these goal-oriented bounds
$\Xi_{\pm}$ represent data-informed quantities that encapsulate
worst-case scenarios for the errors in misspecifying the
geostatistical model due to epistemic uncertainty.

\section{Efficient computation of hybrid information divergences}%
\label{sec:efficient-sampling-risk}%

The variance of the hybrid information divergence $\Xi$ depends on the
variance of the risk-sensitive hybrid performance measure
$\Lambda_\gamma(c, H^g)$ defined in \cref{eq:risk-sensitive-pm}. As
$\Lambda_\gamma$ has the form of a cumulant generating functional of
the marginal performance measure $H^g$ in \cref{eq:hybrid-pm}, the
variance of $\Lambda_\gamma$ behaves like the variance of an
exponential function of the random variable $H^g$. If $\Xi$ exhibits
large variance, as observed for the upper bound $\Xi_{+}$ for the goal
function $g_3$ in
\cref{fig:mod1_uq_bounds_ell,fig:experiment1_g3,fig:experiment2_bounds},
then attempting to reduce the overall variance by running additional
simulations for $H^g$ may be infeasible if the solver for the forward
model is computationally expensive. Presently, we outline a different
strategy that involves estimating $\Lambda_\gamma$ by concentration
inequalities from Large Deviations Theory
(\cite{BoucheronEtAl:2013ci}). This approach was recently introduced
and applied to model problems in \cite{GourgouliasEtAl:2017aa} and we
extend these ideas to complex systems involving random PDE where the
stochastic fields are infinite dimensional and inferred from
real-world data. The concentration inequality estimates that we
explore rely on statistics of $H^g$, such as mean and variance, and
provide a bound on $\Lambda_\gamma$, that can be used as a surrogate
in $\Xi$, accounting for any available data. Thus, not only are these
concentration inequality estimates computationally non-intrusive but
they also rely on quantities that one is already likely to compute in
the normal course of a simulation.

\subsection{Variance of the standard estimator for $\Xi$}%
\label{sec:vari-estimator-xi+}%

The variance of the standard MC estimator for $\Xi_{+}$ is given by
\begin{equation*}
  \var_\gamma\left[\Xi_{+}(\lambda\mid\gamma; H^g)\right] 
  = \frac{\var_\gamma [ e^{c^*H^g}]}{(c^*)^2 M (\E_\gamma [e^{c^*H^g}])^2},
\end{equation*}
a quantity that depends exponentially on both $c^*$ and $H^g$. We
recall that the optimal $c^*$ is linked to the information budget
$\rho = \RE(\lambda \mid \gamma)$ by (\ref{eq:optimal-cstar-rho}). For
alternative models that are close in relative entropy to the nominal
model, such as small parametric perturbations, then
\cref{eq:optimal-cstar-rho} provides a good approximation of the
optimal $c^*$ up to first order in $\rho$. However, for alternative
models that are a large relative entropy distance from the nominal
model, we see from \cref{eq:optimal-cstar-rho} that the optimal $c^*$
grows at least linearly in $\rho$.

Attempting to sample an estimator with large variance poses a
difficulty for the present application of interest as sampling
involves calls to a random PDE solver. In such settings, it is
therefore of interest to find an alternative strategy to sampling
$\Lambda_\gamma$. As suggested by the right-hand side of
\cref{eq:risk-sensitive-pm}, $\Lambda_\gamma$ may have a known
description as a cumulant generating functional for particular $H^g$.
In other instances, $\gamma$ might have a form amenable to the
numerical integration of $\E_{\gamma}[e^{c(H^g-\E_\gamma[H^g])}]$, for
example via thermodynamic integration techniques
(\cite{LelievreRoussetStoltz:2010fe}).

In the remainder of this section, we indicate an alternative approach,
recently introduced in \cite{GourgouliasEtAl:2017aa}, that relies on
concentration inequalities from large deviations theory to bound
$\Lambda_\gamma$. The concentration inequalities, at least in their
simplest form, require bounded observables $H^g$ but rely on
quantities that we are already likely to be sampling in our simulation
such as the expected value and the variance. Although in the form of
concentration inequalities discussed below the observable must be
bounded, we show that such bounds produce fairly reliable results even
when the observable is merely finite and an artificial bound is
imposed (see \cref{rmk:finite-observables}). We refer to
\cite{GourgouliasEtAl:2017aa} for a complete discussion on UQ methods
based on concentration inequalities for both bounded and unbounded
observables in several model problems.

\subsection{Concentration inequalities for risk-sensitive performance
  measures}%
\label{sec:conc-ineq}%

We recall the following bound on the moment generating function of a
random variable in terms of its first two moments (see e.g.\
\cite{DemboZeitouni:2010ld}).

\begin{lemma}[Bennett]
  \label{lem:bennett}
  Supposed $X \leq b$ is a real-valued random variable with
  $m = \E[X]$ and $\E[(X-m)^2] \leq s^2$ for some $s > 0$. Then, for
  any $c \geq 0$,
  \begin{equation*}
    \label{eq:bennett}
    \E[e^{c X}] \leq e^{c m} 
    \left( \frac{(b-m)^2}{(b-m)^2 +s^2}e^{- \frac{cs^2}{b-m}} 
      + \frac{s^2}{(b-m)^2+s^2}e^{c(b-m)} \right).
  \end{equation*}
\end{lemma}

Thus we formulate a bound for $\Lambda_\gamma (c,H^g)$ where the
estimator of this quantity does not involve sampling an exponentially
large quantity, i.e.\ the moment generating functional.

\begin{theorem}[Concentration]
  \label{thm:concentration}
  For a bounded observable $H^g \leq \bar{b}$ and $c \geq 0$,
  \begin{equation*}
    \Lambda_\gamma(c, H^g) \leq 
    \frac{1}{c} \log \left( \frac{(\bar{b}-\widehat{H^g})^2}
      {(\bar{b}-\widehat{H^g})^2+s_g^2} 
      e^{-cs_g^2/(\bar{b}-\widehat{H^g})} 
      + \frac{s_g^2}{(\bar{b}-\widehat{H^g})^2+s_g^2}
      e^{c(\bar{b}-\widehat{H^g})} \right),
  \end{equation*}
  where $\widehat{H^g} = \E_\gamma [H^g]$ and
  $s_g^2 = \var_\gamma[H^g]$.
\end{theorem}

\begin{proof}
  This follows immediately from \cref{lem:bennett} by considering the
  centered $X = H^g - \widehat{H^g}$ with
  $b = \bar{b} - \widehat{H^g}$, $m = \E[H^g - \widehat{H^g}] = 0$,
  and $s_g^2 = \E[X^2] = \var_\gamma[H^g]$.
\end{proof}

We note from \cref{thm:uq-interval-failure}, that for failure
probabilities the risk-sensitive performance measure has the form
\cref{eq:Lambda-failure-prob} thus the bound appearing in
\cref{thm:concentration} holds with equality. An immediate extension
to \cref{lem:bennett} bounds the moment generating function in terms
of its mean and support and can be used when $H^g$ has both an upper
and lower bound.

\begin{lemma}[Bennett-$(a,b)$]
  \label{lem:bennett-ab}
  Suppose $X \in [a,b]$, for fixed $a<b$, is a real-valued random
  variable with $m = \E[X]$. Then for any $c \in \rset$,
  \begin{equation*}
    \E[e^{c X}] \leq \frac{m-a}{b-a} e^{cb} + \frac{b-m}{b-a} e^{ca}.
  \end{equation*}
\end{lemma}

We end by demonstrating these alternative bounds for the experiment in
\cref{sec:screening-and-sa}.

\subsection{Implementation for a parametric model}%
\label{sec:conc-ineq-for-param-model}%

In the spirit of \cref{eq:estimate-xi}, we let
\begin{equation}
  \label{eq:conc-estimate}
  \zeta (c, H^g) := \bar{\Lambda}_\gamma(c, H^g) 
  + \frac{1}{c}\bar{\RE}(\gamma' \mid \gamma),
\end{equation}
and obtain
\begin{equation}
  \label{eq:risk-sensitive-bound-bennett}
  B_{+}(H^g) \approx \zeta (c^*, H^g) \qquad \text{and} 
  \qquad C_{+}(H^g) \approx \zeta (c^*, H^g),
\end{equation}
for an optimal $c^*$ where $\bar{\Lambda}_\gamma(c ,H^g)$ in
\cref{eq:conc-estimate} is approximated using \cref{thm:concentration}
and \cref{lem:bennett-ab}, respectively. These surrogates for the
cumulant generating functional account for any available data through
suitable statistical quantities of the marginal performance measure
$H^g$ in \cref{eq:hybrid-pm} that is related to the propagation of
model-form uncertainty from the geostatistical model to the forward
model in \cref{fig:layers-subsurface-flow-model}. Corresponding lower
bounds, $B_{-}(H^g)$ and $C_{-}(H^g)$, are derived in a similar
manner. In \cref{fig:mod1_g3_xi_and_bennett}, we demonstrate all of
the bounds for the goal functional $g_3$
\cref{eq:data-goal-functionals-g3}, noting that the bounds $B_{\pm}$
and $C_{\pm}$ have much smaller variance than the estimator for
$\Xi_{\pm}$ and in each case form an envelope around the sensitivity
for the QoI that remains tight and robust. In
\cref{fig:mod1_g3_xi_and_bennett}, the shaded regions denoting
confidence intervals for the concentration inequalities $B_\pm$ and
$C_\pm$ are negligible compared to the shaded region for the estimator
of $\Xi_\pm$. The smaller variance of the concentration inequalities
comes at the cost of some tightness in the estimates for
$\mathcal{E}$.

\begin{figure}[]
  \centering
  \subfloat[]{\label{fig:fig:mod1_g3_xi_and_bennett-ell}
    \includegraphics[scale=1]{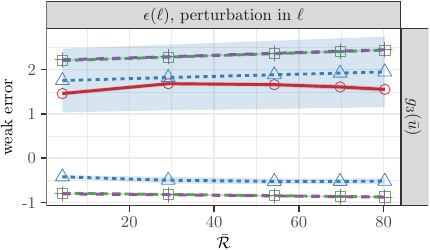}} \hfill
  \subfloat[]{\label{fig:fig:mod1_g3_xi_and_bennett-tau}
    \includegraphics[scale=1.025]{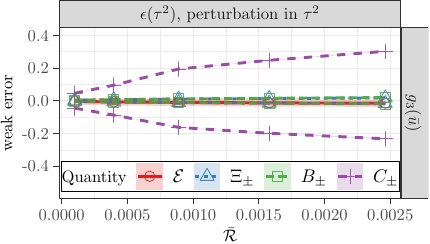}}
  \caption{The bounds $\Xi_{\pm}$, $B_\pm$, and $C_\pm$ are compared
    above for predictions of the weak error $\mathcal{E}$ between
    small parametric perturbations of $\ell$ (a) and $\tau^2$ (b) for
    the parametric geostatistical model with mean $\mu$ and covariance
    \cref{eq:cov-se-nugget} (cf.\
    \cref{fig:mod1_uq_bounds_ell,fig:mod1_uq_bounds_tau}). The bounds
    $B_\pm$ and $C_\pm$ in \cref{eq:risk-sensitive-bound-bennett}
    based on the Bennett and Bennett-$(a,b)$ concentration
    inequalities in \cref{lem:bennett,lem:bennett-ab} provide a
    computationally efficient alternative to sampling the cumulant
    generating functional for the marginal performance measure
    directly when the goal functional has high variance as is the case
    for \cref{eq:data-goal-functionals-g3}. Notice in (a) that the
    shaded regions denoting confidence intervals for the concentration
    inequalities $B_\pm$ and $C_\pm$ are negligible compared to shaded
    region for the estimator of $\Xi_\pm$. The smaller variance of the
    concentration inequalities comes at the cost of some tightness in
    the estimates for $\mathcal{E}$; in this case
    $\Xi \leq B_\pm \leq C_\pm$. In particular, we observe in (b) that
    $C_\pm$ over predicts the weak error $\mathcal{E}$.}
  \label{fig:mod1_g3_xi_and_bennett}
\end{figure}

\begin{remark}[Unbounded QoIs]
  \label{rmk:finite-observables}
  Although the concentration inequalities as quoted here are indicated
  only for a bounded QoI, we note that there exist other formulations
  for unbounded QoIs such as for sub-Gaussian random variables
  (\cite{GourgouliasEtAl:2017aa}). In practice
  \cref{thm:concentration} is a useful computational tool for a finite
  QoI; we observe the tight bounds demonstrated in
  \cref{fig:mod1_g3_xi_and_bennett} are for
  $g_3(\bar{u}) = \min(\bar{u}(1), 3)$ where the cut-off was
  arbitrarily chosen using a training set of \num{e3} observations of
  $\bar{u}(1)$.
\end{remark}

\section{Conclusions}%

The present work develops UQ tools for a random PDE model of
steady-state subsurface flow in
\cref{fig:layers-subsurface-flow-model} with potential impacts in
hydrology, carbon sequestration, and petroleum engineering. These
tools are realized through the novel application of hybrid information
divergences that balance observable and data dependent quantities. The
hybrid nature of the divergences allows us to represent and
distinguish various sources of uncertainty entering into the model by
attaching different levels of confidence to parts of the model.
Ultimately, this allows us to address a key challenge concerning the
propagation of model-form or epistemic uncertainty from the
geostatistical model via the pathway in
\cref{fig:propagation-uncertainty}.

We derive tight and robust estimates for modeling errors or biases
from the hybrid information divergences and apply these to important
UQ tasks including parametric sensitivity analysis and model
misspecification arising from sparse data. In particular, we
demonstrate the use of these bounds for making data-informed
predictions such as quantifying the impact of incomplete data as in
\cref{sec:data-informed-bounds}. The robustness, when interpreted as
including worst-case scenarios within a given information budget
(i.e.\ within a family of acceptable alternative models, see
\cref{fig:bounds-across-family}), suggests that these bounds are an
appropriate deliverable in the context of the decision support
framework in \cref{fig:layers-subsurface-flow-model}. We emphasize
that the bounds derived here are also goal-oriented and non-intrusive
in nature, that is, can be used in conjunction with any algorithm or
solver for the random PDE problem in
\cref{fig:layers-subsurface-flow-model}. Finally, we also make
connections between the hybrid information divergences and certain
concentration inequalities from Large Deviations Theory that can be
leveraged for efficient computing and account for any available data
through suitable statistical quantities.

\appendix 

\section{Source code}

Source code and links to observational data are available at\\
\url{https://github.com/ejhall/robust-uq-divergences}.

\bibliographystyle{siamplain}%
\bibliography{uq_rpde_r2}%

\begin{thebibliography}{10}

\bibitem{AarnesEtAl:2009ms}
{\sc J.~r.~E. Aarnes, K.-A. Lie, V.~Kippe, and S.~Krogstad}, {\em Multiscale
  methods for subsurface flow}, in Multiscale {M}odeling and {S}imulation in
  {S}cience, B.~Engquist, P.~L{\"o}stedt, and O.~Runborg, eds., vol.~66,
  Springer, Berlin, 2009, pp.~3--48,
  \url{https://doi.org/10.1007/978-3-540-88857-4_1}.

\bibitem{AndersonEtAl:2015gw}
{\sc M.~P. Anderson, W.~W. Woessner, and R.~J. Hunt}, {\em Applied
  {G}roundwater {M}odeling: {S}imulation of {F}low and {A}dvective
  {T}ransport}, Academic Press, 2015,
  \url{https://doi.org/10.1016/B978-0-08-091638-5.00013-4}.

\bibitem{ArampatzisEtAl:2015as}
{\sc G.~Arampatzis, M.~A. Katsoulakis, and Y.~Pantazis}, {\em Accelerated
  sensitivity analysis in high-dimensional stochastic reaction networks}, PLoS
  ONE, 10 (2015), e0130825, \url{https://doi.org/10.1371/journal.pone.0130825}.

\bibitem{AtarChowdharyDupuis:2015rd}
{\sc R.~Atar, K.~Chowdhary, and P.~Dupuis}, {\em Robust bounds on
  risk-sensitive functionals via r{\'e}nyi divergence}, SIAM/ASA J. Uncertain.
  Quantif., 3 (2015), pp.~18--33, \url{https://doi.org/10.1137/130939730}.

\bibitem{BabuskaNobileTempone:2007}
{\sc I.~Babu{\v{s}}ka, F.~Nobile, and R.~Tempone}, {\em A stochastic
  collocation method for elliptic partial differential equations with random
  input data}, SIAM J. Numer. Anal., 45 (2007), pp.~1005--1034,
  \url{https://doi.org/10.1137/050645142}.

\bibitem{Ben-TalGhaouiNemirovski:2009ro}
{\sc A.~Ben-Tal, L.~El~Ghaoui, and A.~Nemirovski}, {\em Robust Optimization},
  Princeton Series in Applied Mathematics, Princeton University Press, October
  2009, \url{https://doi.org/10.1515/9781400831050}.

\bibitem{BoucheronEtAl:2013ci}
{\sc S.~Boucheron, G.~Lugosi, and P.~Massart}, {\em Concentration Inequalities:
  A Nonasymptotic Theory of Independence}, Oxford University Press, 2013,
  \url{https://doi.org/10.1093/acprof:oso/9780199535255.001.0001}.

\bibitem{BurnhamAnderson:2002ms}
{\sc K.~P. Burnham and D.~R. Anderson}, {\em Model {S}election and {M}ultimodel
  {I}nference}, Springer-Verlag, New York, second~ed., 2002,
  \url{https://doi.org/10.1007/b97636}.

\bibitem{Charrier:2012}
{\sc J.~Charrier}, {\em Strong and weak error estimates for elliptic partial
  differential equations with random coefficients}, SIAM J. Numer. Anal., 50
  (2012), pp.~216--246, \url{https://doi.org/10.1137/100800531}.

\bibitem{CharrierScheichlTeckentrup:2013}
{\sc J.~Charrier, R.~Scheichl, and A.~L. Teckentrup}, {\em Finite element error
  analysis of elliptic {PDE}s with random coefficients and its application to
  multilevel {M}onte {C}arlo methods}, SIAM J. Numer. Anal., 51 (2013),
  pp.~322--352, \url{https://doi.org/10.1137/110853054}.

\bibitem{DupuisChowdhary:2013ae}
{\sc K.~Chowdhary and P.~Dupuis}, {\em Distinguishing and integrating aleatoric
  and epistemic variation in uncertainty quantification}, ESAIM Math. Model.
  Numer. Anal., 47 (2013), pp.~635--662,
  \url{https://doi.org/10.1051/m2an/2012038}.

\bibitem{ChristieBlunt:2001sp}
{\sc M.~Christie, M.~Blunt, et~al.}, {\em Tenth {SPE} comparative solution
  project: A comparison of upscaling techniques}, in SPE Reservoir Simulation
  Symposium, Society of Petroleum Engineers, 2001,
  \url{https://doi.org/10.2118/72469-PA}.

\bibitem{CliffeGilesScheichlTeckentrup:2011}
{\sc K.~A. Cliffe, M.~B. Giles, R.~Scheichl, and A.~L. Teckentrup}, {\em
  Multilevel {M}onte {C}arlo methods and applications to elliptic {PDE}s with
  random coefficients}, Comput. Vis. Sci., 14 (2011), pp.~3--15,
  \url{https://doi.org/10.1007/s00791-011-0160-x}.

\bibitem{CoverThomas:2006in}
{\sc T.~M. Cover and J.~A. Thomas}, {\em Elements of Information Theory},
  Wiley-Interscience, Hoboken, NJ, second~ed., 2006,
  \url{https://doi.org/10.1002/047174882X}.

\bibitem{Dagan:1986}
{\sc G.~Dagan}, {\em Statistical theory of groundwater flow and transport: Pore
  to laboratory, laboratory to formation, and formation to regional scale},
  Water Resources Research, 22 (1986), pp.~120S--134S,
  \url{https://doi.org/10.1029/WR022i09Sp0120S}.

\bibitem{Dagan1989:ft}
{\sc G.~Dagan}, {\em Flow and {T}ransport in {P}orous {F}ormations},
  Springer-Verlag, Berlin, 1989,
  \url{https://doi.org/10.1007/978-3-642-75015-1}.

\bibitem{DemboZeitouni:2010ld}
{\sc A.~Dembo and O.~Zeitouni}, {\em Large {D}eviations {T}echniques and
  {A}pplications}, vol.~38 of Stochastic Modelling and Applied Probability,
  Springer-Verlag, Berlin, 2010,
  \url{https://doi.org/10.1007/978-3-642-03311-7}.

\bibitem{DohertyHunt:2010ap}
{\sc J.~E. Doherty and R.~J. Hunt}, {\em Approaches to highly parameterized
  inversion: {A} guide to using {PEST} for model-parameter and
  predictive-uncertainty analysis}, tech. report, US Geological Survey, 2010,
  \url{http://pubs.usgs.gov/sir/2010/5211}.

\bibitem{DuchiJordanWainwright:2013lp}
{\sc J.~C. Duchi, M.~I. Jordan, and M.~J. Wainwright}, {\em Local privacy and
  statistical minimax rates}, in 2013 IEEE 54th Annual Symposium on Foundations
  of Computer Science, Oct 2013, pp.~429--438,
  \url{https://doi.org/10.1109/FOCS.2013.53}.

\bibitem{DupuisEllis:1997ld}
{\sc P.~Dupuis and R.~S. Ellis}, {\em A {W}eak {C}onvergence {A}pproach to the
  {T}heory of {L}arge {D}eviations}, John Wiley \& Sons, New York, 1997,
  \url{https://doi.org/10.1002/9781118165904}.

\bibitem{DupuisEtAl:2015ps}
{\sc P.~Dupuis, M.~A. Katsoulakis, Y.~Pantazis, and P.~Plech{{\'a}}{\v{c}}},
  {\em Path-space information bounds for uncertainty quantification and
  sensitivity analysis of stochastic dynamics}, SIAM/ASA J. Uncertain.
  Quantif., 4 (2016), pp.~80--111, \url{https://doi.org/10.1137/15M1025645}.

\bibitem{Durner:1994hc}
{\sc W.~Durner}, {\em Hydraulic conductivity estimation for soils with
  heterogeneous pore structure}, Water Resour. Res., 30 (1994), pp.~211--223,
  \url{https://doi.org/10.1029/93WR02676}.

\bibitem{Ellis:2006ed}
{\sc R.~S. Ellis}, {\em Entropy, {L}arge {D}eviations, and {S}tatistical
  {M}echanics}, Springer-Verlag, Berlin, 2006,
  \url{https://doi.org/10.1007/3-540-29060-5},
  \url{http://dx.doi.org/10.1007/3-540-29060-5}.

\bibitem{ErnstEtAl:2012pc}
{\sc O.~G. Ernst, A.~Mugler, H.-J. Starkloff, and E.~Ullmann}, {\em On the
  convergence of generalized polynomial chaos expansions}, ESAIM Math. Model.
  Numer. Anal., 46 (2012), pp.~317--339,
  \url{https://doi.org/10.1051/m2an/2011045}.

\bibitem{Ewing:1983rs}
{\sc R.~E. Ewing}, ed., {\em The {M}athematics of {R}eservoir {S}imulation},
  vol.~1 of Frontiers in Applied Mathematics, SIAM, Philadelphia, PA, 1983,
  \url{https://doi.org/10.1137/1.9781611971071}.

\bibitem{FersonGinzburg:1996uq}
{\sc S.~Ferson and L.~R. Ginzburg}, {\em Different methods are needed to
  propagate ignorance and variability}, Reliab. Eng. Syst. Safe., 54 (1996),
  pp.~133--144,
  \url{https://doi.org/http://dx.doi.org/10.1016/S0951-8320(96)00071-3}.

\bibitem{GelfandEtAl:2010hb}
{\sc A.~E. Gelfand, P.~J. Diggle, M.~Fuentes, and P.~Guttorp}, eds., {\em
  Handbook of Spatial Statistics}, CRC Press, Boca Raton, FL, 2010,
  \url{https://doi.org/10.1201/9781420072884}.

\bibitem{GilEtAl:2013rd}
{\sc M.~Gil, F.~Alajaji, and T.~Linder}, {\em R{\'e}nyi divergence measures for
  commonly used univariate continuous distributions}, Inform. Sci., 249 (2013),
  pp.~124--131, \url{https://doi.org/10.1016/j.ins.2013.06.018}.

\bibitem{Glasserman:2003mc}
{\sc P.~Glasserman}, {\em Monte Carlo Methods in Financial Engineering},
  vol.~53, Springer, New York, 2003,
  \url{https://doi.org/10.1007/978-0-387-21617-1}.

\bibitem{GlassermanYao:1992gg}
{\sc P.~Glasserman and D.~D. Yao}, {\em Some guidelines and guarantees for
  common random numbers}, Management Science, 38 (1992), pp.~884--908,
  \url{https://doi.org/10.1287/mnsc.38.6.884}.

\bibitem{GourgouliasEtAl:2016sp}
{\sc K.~Gourgoulias, M.~A. Katsoulakis, and L.~Rey-Bellet}, {\em Information
  metrics for long-time errors in splitting schemes for stochastic dynamics and
  parallel kinetic monte carlo}, SIAM J. Sci. Comput., 38 (2016),
  pp.~A3808--A3832, \url{https://doi.org/10.1137/15M1047271}.

\bibitem{GourgouliasEtAl:2017aa}
{\sc K.~Gourgoulias, M.~A. Katsoulakis, L.~Rey-Bellet, and J.~Wang}, {\em How
  biased is your model? {C}oncentration inequalities, information and model
  bias}.
\newblock Sumbitted (arXiv:1706.10260), Jun 2017,
  \url{https://arxiv.org/abs/1706.10260}.

\bibitem{Halletal:2016sc}
{\sc E.~J. Hall, H.~Hoel, M.~Sandberg, A.~Szepessy, and R.~Tempone}, {\em
  Computable error estimates for finite element approximations of elliptic
  partial differential equations with rough stochastic data}, SIAM J. Sci.
  Comput., 38 (2016), pp.~A3773--A3807,
  \url{https://doi.org/10.1137/15M1044266}.

\bibitem{HarmandarisEtAl:2016vi}
{\sc V.~Harmandaris, E.~Kalligiannaki, M.~Katsoulakis, and P.~Plech{\'a}{\v
  c}}, {\em Path-space variational inference for non-equilibrium coarse-grained
  systems}, J. Comput. Phys., 314 (2016), pp.~355 -- 383,
  \url{https://doi.org/http://dx.doi.org/10.1016/j.jcp.2016.03.021}.

\bibitem{Helton:1994uq}
{\sc J.~C. Helton}, {\em Treatment of uncertainty in performance assessments
  for complex systems}, Risk Anal., 14 (1994), pp.~483--511,
  \url{https://doi.org/10.1111/j.1539-6924.1994.tb00266.x}.

\bibitem{HoffmanHammonds:1994pu}
{\sc F.~O. Hoffman and J.~S. Hammonds}, {\em Propagation of uncertainty in risk
  assessments: The need to distinguish between uncertainty due to lack of
  knowledge and uncertainty due to variability}, Risk Anal., 14 (1994),
  pp.~707--712, \url{https://doi.org/10.1111/j.1539-6924.1994.tb00281.x}.

\bibitem{Hora:1996ae}
{\sc S.~C. Hora}, {\em Aleatory and epistemic uncertainty in probability
  elicitation with an example from hazardous waste management}, Reliab. Eng.
  Syst. Safe., 54 (1996), pp.~217--223,
  \url{https://doi.org/10.1016/S0951-8320(96)00077-4}.

\bibitem{KatsoulakisEtAl:2017sc}
{\sc M.~A. Katsoulakis, L.~Rey-Bellet, and J.~Wang}, {\em Scalable information
  inequalities for uncertainty quantification}, J. Comput. Phys., 336 (2017),
  pp.~513 -- 545, \url{https://doi.org/10.1016/j.jcp.2017.02.020}.

\bibitem{KomorowskiEtAl:2011sa}
{\sc M.~Komorowski, M.~J. Costa, D.~A. Rand, and M.~P. Stumpf}, {\em
  Sensitivity, robustness, and identifiability in stochastic chemical kinetics
  models}, Proc. Natl. Acad. Sci. USA, 108 (2011), pp.~8645--8650,
  \url{https://doi.org/10.1073/pnas.1015814108}.

\bibitem{LelievreRoussetStoltz:2010fe}
{\sc T.~Leli{\`e}vre, M.~Rousset, and G.~Stoltz}, {\em Free Energy
  Computations}, Imperial College Press, London, 2010,
  \url{https://doi.org/10.1142/9781848162488}.

\bibitem{LiQiXiu:2014uq}
{\sc J.~Li, X.~Qi, and D.~Xiu}, {\em On upper and lower bounds for quantity of
  interest in problems subject to epistemic uncertainty}, SIAM J. Sci. Comput.,
  36 (2014), pp.~A364--A376, \url{https://doi.org/10.1137/120892969}.

\bibitem{LiXiu:2012fp}
{\sc J.~Li and D.~Xiu}, {\em Computation of failure probability subject to
  epistemic uncertainty}, SIAM J. Sci. Comput., 34 (2012), pp.~A2946--A2964,
  \url{https://doi.org/10.1137/120864155}.

\bibitem{LieseVajda:2007sd}
{\sc F.~Liese and I.~Vajda}, {\em Convex Statistical Distances}, Teubner,
  Leipzig, 2007.

\bibitem{MatthiesKeese:2005}
{\sc H.~G. Matthies and A.~Keese}, {\em Galerkin methods for linear and
  nonlinear elliptic stochastic partial differential equations}, Comput.
  Methods Appl. Mech. Engrg., 194 (2005), pp.~1295--1331,
  \url{https://doi.org/10.1016/j.cma.2004.05.027}.

\bibitem{McLaughlinTownley:1996ip}
{\sc D.~McLaughlin and L.~R. Townley}, {\em A reassessment of the groundwater
  inverse problem}, Water Resour. Res., 32 (1996), pp.~1131--1161,
  \url{https://doi.org/10.1029/96WR00160}.

\bibitem{OberkampfEtAl:2004cp}
{\sc W.~L. Oberkampf, J.~C. Helton, C.~A. Joslyn, S.~F. Wojtkiewicz, and
  S.~Ferson}, {\em Challenge problems: {U}ncertainty in system response given
  uncertain parameters}, Reliab. Eng. Syst. Safe., 85 (2004), pp.~11--19,
  \url{https://doi.org/10.1016/j.ress.2004.03.002}.

\bibitem{OwhadiEtAl:2013uq}
{\sc H.~Owhadi, C.~Scovel, T.~J. Sullivan, M.~McKerns, and M.~Ortiz}, {\em
  Optimal uncertainty quantification}, SIAM Review, 55 (2013), pp.~271--345,
  \url{https://doi.org/10.1137/10080782X},
  \url{https://doi.org/10.1137/10080782X},
  \url{https://arxiv.org/abs/https://doi.org/10.1137/10080782X}.

\bibitem{Parry:1996uq}
{\sc G.~W. Parry}, {\em The characterization of uncertainty in probabilistic
  risk assessments of complex systems}, Reliab. Eng. Syst. Safe., 54 (1996),
  pp.~119--126, \url{https://doi.org/10.1016/S0951-8320(96)00069-5}.

\bibitem{PateCornell:1996uq}
{\sc M.~E. Pat{\'e}-Cornell}, {\em Uncertainties in risk analysis: Six levels
  of treatment}, Reliab. Eng. Syst. Safe., 54 (1996), pp.~9--111,
  \url{https://doi.org/10.1016/S0951-8320(96)00067-1}.

\bibitem{RasmussenWilliams:2006gp}
{\sc C.~E. Rasmussen and C.~K.~I. Williams}, {\em Gaussian {P}rocesses for
  {M}achine {L}earning}, MIT Press, Cambridge, MA, 2006,
  \url{http://www.gaussianprocess.org/gpml/}.

\bibitem{Richards:1992bg}
{\sc P.~Richards}, {\em An introduction to the {B}rent {G}roup: A literature
  review}, in Geology of the Brent Group, A.~Morton, R.~Haszeldine, M.~Giles,
  and S.~Brown, eds., no.~61 in Special Publications, Geological Society of
  London, London, 1992, pp.~15--26,
  \url{https://doi.org/10.1144/GSL.SP.1992.061.01.03}.

\bibitem{Rowe:1994uu}
{\sc W.~D. Rowe}, {\em Understanding uncertainty}, Risk Anal., 14 (1994),
  pp.~743--750, \url{https://doi.org/10.1111/j.1539-6924.1994.tb00284.x}.

\bibitem{RueHeld:2005}
{\sc H.~Rue and L.~Held}, {\em Gaussian Markov Random Fields}, vol.~104 of
  Monographs on Statistics and Applied Probability, Chapman and Hall/CRC, Boca
  Raton, FL, 2005, \url{https://doi.org/10.1201/9780203492024}.

\bibitem{SchlatherEtAl:2015rf}
{\sc M.~Schlather, A.~Malinowski, M.~Oesting, D.~Boecker, K.~Strokorb,
  S.~Engelke, J.~Martini, F.~Ballani, O.~Moreva, J.~Auel, P.~Menck, S.~Gross,
  U.~Ober, C.~Berreth, K.~Burmeister, J.~Manitz, P.~Ribeiro, R.~Singleton,
  B.~Pfaff, and {R Core Team}}, {\em Randomfields: {S}imulation and {A}nalysis
  of {R}andom {F}ields}, 2017,
  \url{https://cran.r-project.org/web/packages/RandomFields/index.html}
  (accessed 2017-04-18).
\newblock R package version 3.1.50.

\bibitem{Tartakovsky2013:rv}
{\sc D.~M. Tartakovsky}, {\em Assessment and management of risk in subsurface
  hydrology: A review and perspective}, Advances in Water Resources, 51 (2013),
  pp.~247--260, \url{https://doi.org/10.1016/j.advwatres.2012.04.007}.

\bibitem{TsourtisEtAl:2015}
{\sc A.~Tsourtis, Y.~Pantazis, M.~A. Katsoulakis, and V.~Harmandaris}, {\em
  Parametric sensitivity analysis for stochastic molecular systems using
  information theoretic metrics}, J. Chem. Phys., 143 (2015), 014116,
  \url{https://doi.org/10.1063/1.4922924}.

\bibitem{Tsybakov:2009np}
{\sc A.~B. Tsybakov}, {\em Introduction to Nonparametric Estimation}, Springer
  Series in Statistics, Springer, New York, 2009,
  \url{https://doi.org/10.1007/978-0-387-79052-7}.

\bibitem{Wasserman:2013as}
{\sc L.~Wasserman}, {\em All of {S}tatistics}, Springer-Verlag, New York, 2004,
  \url{https://doi.org/10.1007/978-0-387-21736-9}.

\end{thebibliography}

\end{document}